\documentclass{amsart}

\usepackage{}
\usepackage{amssymb}
\usepackage[colorlinks=true, linkcolor=blue, citecolor=red]{hyperref}
\usepackage{bbm}
\usepackage{cases}
\usepackage{amsmath}
\usepackage{graphicx}
\usepackage{mathrsfs}
\usepackage{stmaryrd}
\usepackage{color}
\usepackage{soul}
\usepackage[dvipsnames]{xcolor}
\usepackage{amsfonts}
\usepackage{enumerate,amsmath,amssymb,amsthm}

\numberwithin{equation}{section}

\newcommand{\be}{\begin{eqnarray}}
\newcommand{\mE}{\end{eqnarray}}
\newcommand{\ce}{\begin{eqnarray*}}
\newcommand{\de}{\end{eqnarray*}}
\newtheorem{theorem}{Theorem}[section]
\newtheorem{lemma}[theorem]{Lemma}
\newtheorem{remark}[theorem]{Remark}
\newtheorem{definition}[theorem]{Definition}
\newtheorem{proposition}[theorem]{Proposition}
\newtheorem{example}[theorem]{Example}
\newtheorem{corollary}[theorem]{Corollary}

\def\eps{\varepsilon}

\def\p{\partial}

\def\[{{\Big[}}
\def\]{{\Big]}}
\def\<{{\langle}}
\def\>{{\rangle}}
\def\({{\Big(}}
\def\){{\Big)}}

\def\bx{{\mathbf{x}}}

\def\dif{{\mathord{{\rm d}}}}

\def\no{\nonumber}
\def\={&\!\!=\!\!&}
\def\bt{\begin{theorem}}
\def\et{\end{theorem}}
\def\bl{\begin{lemma}}
\def\el{\end{lemma}}
\def\br{\begin{remark}}
\def\er{\end{remark}}

\def\bd{\begin{definition}}
\def\ed{\end{definition}}
\def\bp{\begin{proposition}}
\def\ep{\end{proposition}}
\def\bc{\begin{corollary}}
\def\ec{\end{corollary}}
\def\bx{\begin{example}}
\def\ex{\end{example}}
\def\cA{{\mathcal A}}

\def\cD{{\mathcal D}}

\def\cF{{\mathcal F}}
\def\cG{{\mathcal G}}
\def\cH{{\mathcal H}}

\def\cL{{\mathcal L}}

\def\cZ{{\mathcal Z}}

\def\mC{{\mathbb C}}

\def\mE{{\mathbb E}}

\def\mN{{\mathbb N}}

\def\mP{{\mathbb P}}

\def\mR{{\mathbb R}}

\def\sI{{\mathscr I}}
\def\sJ{{\mathscr J}}

\def\sL{{\mathscr L}}

\def\sN{{\mathscr N}}

\def\sQ{{\mathscr Q}}

\def\sT{{\mathscr T}}
\def\sU{{\mathscr U}}

\def\sW{{\mathscr W}}

\def\sY{{\mathscr Y}}

\def\geq{\geqslant}
\def\leq{\leqslant}

\begin{document}

\title[Multi-scale stochastic hyperbolic-parabolic equations]{Averaging principle and normal deviations for multi-scale stochastic hyperbolic-parabolic equations}


\author{Michael R\"ockner}
\address{
	Fakult\"{a}t f\"{u}r Mathematik, Universit\"{a}t Bielefeld, D-33501 Bielefeld, Germany, and Academy of Mathematics and Systems Science,
	Chinese Academy of Sciences (CAS), Beijing, 100190, P.R.China}
\curraddr{}
\email{roeckner@math.uni-bielefeld.de}
\thanks{}

\author{Longjie Xie}
\address{School of Mathematics and Statistics $\&$ Research Institute of Mathematical Science, Jiangsu Normal University,
	Xuzhou, Jiangsu 221000, P.R.China}
\curraddr{}
\email{longjiexie@jsnu.edu.cn}
\thanks{This work is supported  by the DFG through CRC 1283 and NSFC (No. 12071186, 11931004).}

\author{Li Yang}
\address{Department of Mathematics, Shandong University,
	Jinan, Shandong 250100, P.R.China}
\curraddr{}
\email{llyang@mail.sdu.edu.cn}
\thanks{}

\subjclass[2020]{  60H15, 60F05, 70K70}
\keywords{Stochastic hyperbolic-parabolic equations; averaging principle; strong and weak convergence; homogenization}

\date{}

\dedicatory{}

\begin{abstract}
We  study the asymptotic behavior of stochastic hyperbolic parabolic equations with slow and fast time scales.  Both the  strong and weak convergence in the averaging principe are established, which can be viewed as a functional law of large numbers. Then we  study the stochastic fluctuations of the original system around its averaged equation. We show that the normalized difference converges weakly to the solution of a linear stochastic wave equation, which is a form of functional central limit theorem. We provide a unified proof for the above convergence by using the Poisson equation in Hilbert spaces.  Moreover, sharp rates of convergence  are obtained, which are shown not to depend on the regularity of the coefficients in the equation for the fast variable.
\end{abstract}

\maketitle
\section{Introduction}

Let $T>0$ and $D\subseteq\mR^d\;(d\geq 1)$ be a bounded open set. Consider the following system of  stochastic hyperbolic-parabolic equations:
\begin{equation} \label{spde01}
\left\{ \begin{aligned}
&\frac{\partial^2U^{\eps}_t(\xi)}{\partial t^2}=\Delta U^{\eps}_t(\xi)+f(U^{\eps}_t(\xi), Y^{\eps}_t(\xi))+\dot W^1_t(\xi),\qquad\,\,\, (t,\xi)\in(0,T]\times D,\\
&\frac{\partial Y^{\eps}_t(\xi)}{\partial t} =\frac{1}{\eps}\Delta Y^{\eps}_t(\xi)+\frac{1}{\eps}g(U^{\eps}_t(\xi), Y^{\eps}_t(\xi))\!+\!\frac{1}{\sqrt{\eps}} \dot W_t^2(\xi),\,\, (t,\xi)\in(0,T]\times D,\\
&U^{\eps}_t(\xi)=Y^{\eps}_t(\xi)=0,\quad\quad\qquad\quad\qquad\quad\qquad\quad\qquad\quad\,\, (t,\xi)\in(0,T]\times\p D,\\
&U^{\eps}_0(\xi)=u(\xi),\,\,\frac{\partial U^{\eps}_t(\xi)}{\partial t}\big|_{t=0}=v(\xi),\,\,Y^{\eps}_0(\xi)=y(\xi),\quad\,\,\quad\quad\quad\quad\qquad\xi\in D,
\end{aligned} \right.
\end{equation}
where  $\Delta$ is the Laplacian operator, $\p D$ denotes the boundary of the domain $D$, $f,g :\mR^2\to\mR$ are measurable functions, $W^1_t$ and $W^2_t$ are two mutually independent $Q_1$- and $Q_2$-Wiener processes both defined on a complete probability space $(\Omega,\mathscr{F},\{\mathscr{F}_t\}_{t\geq0},\mP)$, and the small parameter $0<\eps\ll 1$ represents the  separation  of  time scales  between   the `slow' process  $U_t^\eps$ and the `fast' motion $Y_t^\eps$ (with time order $1/\eps$).
Randomly perturbed hyperbolic partial differential equations are usually used to model wave propagation and mechanical vibration in a random medium, see e.g. \cite{BD,BDT,DKMNX}. If these phenomena are temperature dependent or heat generating, then the underlying hyperbolic equation will be coupled with a stochastic parabolic equation, which leads to the mathematical description of slow-fast systems through (\ref{spde01}), see e.g. \cite{CPL1,Le,RR,ZZ}  and the references therein. In this respect, the question that how a thermal environment at large time scales may influence the dynamics of the whole system arises.

\vspace{1mm}
In the mathematical literature,  powerful averaging and homogenization methods have been developed to study the asymptotic behavior of multi-scale systems  as  $\eps\to 0$.
The averaging principle  can be viewed as a functional law of large numbers, which says the slow component will converge to the solution of the so-called averaged equation as  $\eps\to 0$. The averaged equation then captures the evolution of the original system over a long time scale, which does not depend on the fast variable any more and thus is much simpler. This theory  was first studied  by Bogoliubov \cite{BM} for deterministic ordinary differential equations, and extended to stochastic differential equations (SDEs for short)  by Khasminskii \cite{K1}, see also \cite{BK,GR,HLi,KY,V0} and the references therein. As a rule, the averaging method requires certain smoothness on both the original and
the averaged coefficients of the systems. Various assumptions have been studied in order to guarantee
the above convergence.
Recently, the averaging principle for two time scale stochastic partial differential equations (SPDEs for short) has attracted considerable attention.
In \cite{CF}, Cerrai and Freidlin proved the averaging principle for slow-fast stochastic reaction-diffusion equations  with noise only in the fast motion. Later, Cerrai \cite{Ce,C2} generalized this result to more general reaction-diffusion equations, see also \cite{BYY,CL,RXY,WR} and the references therein for further developments.
We also mention that  Br\'ehier \cite{Br1,Br2}  studied the rate of convergence in terms of $\eps\to 0$ in the averaging principle for parabolic SPDEs and obtained the $1/2$-order rate of strong convergence (in the mean-square sense) and the $1$-order rate of weak convergence (in the distribution sense), which are known to be optimal. These rates of convergence are  important for the study of other limit theorems in probability theory and numerical schemes, known as the Heterogeneous Multi-scale Method for the original multi-scale system, see e.g. \cite{Br3,EL}.
Concerning stochastic hyperbolic-parabolic equations,  Fu ect. \cite{FWLL1} established the strong convergence in the averaging principle for system (\ref{spde01}) when $d=1$ by the classical Khasminskii time discretization method,  and obtained the $1/4$-order rate of strong convergence. In \cite{FWLL}, by using asymptotic expansion arguments, the authors studied  the weak order convergence for system (\ref{spde01}), but only in a not fully coupled case  ($g(u,y)=g(y)$), i.e., the fast equation does not depend on the slow process.

\vspace{1mm}

In this paper, we shall first prove the strong and weak convergence in the averaging principle for the fully coupled system (\ref{spde01}) with singular coefficients, see {\bf Theorem \ref{main1}}.  Compared with \cite{FWLL,FWLL1}, we assume that the coefficients are only $\eta$-H\"older continuous  with respect to the fast variable with any $\eta>0$, and we obtain the optimal   $1/2$-order rate of strong convergence as well as the  $1$-order rate of weak convergence. Moreover, we find that both the strong and weak convergence rates do not  depend on the regularity of the coefficients in the equation for the fast variable. This implies that the evolution of the multi-scale system (\ref{spde01}) relies mainly on the slow variable, which coincides with the intuition since in the limit equation the fast component has been totally averaged out. Furthermore, the arguments we use  are different from those in \cite{Br1,Ce,C2,CF,FWLL,FWLL1}. Our method to establish the strong and weak convergence is based on the Poisson equation in Hilbert space, which is more unifying and much simpler.

   \vspace{1mm}

The averaged equation for (\ref{spde01})  is only valid  in the limit when the time scale separation between the fast and slow variables is infinitely wide. Of course, the scale separation is never infinite in reality.  For small but positive $\eps,$ the slow variable $U_t^\eps$ will experience fluctuations around its averaged motion $\bar U_t$. These small fluctuations can be captured by  studying the  functional central limit theorem. Namely, we are interested in the asymptotic behavior of the normalized difference
\begin{align}\label{zte}
Z_{t}^{\eps}:=\frac{U_t^\eps-\bar U_t}{\sqrt{\eps}}
\end{align}
as $\eps$ tends to 0. Such result is known to be closely related to the homogenization behavior of singularly perturbed partial differential equations, which is of its own interest, see e.g.  \cite{HP,HP2}.
For the study of the functional central limit theorem  for finite dimensional multi-scale systems, we refer the reader to the fundamental paper by Khasminskii \cite{K1},  see also \cite{BK,CFKM,KM,Li,P-V,P-V2,RX}. The infinite dimensional situation is more open and papers on this subject are very few. In \cite{Ce2}, Cerrai studied the normal deviations for a deterministic reaction-diffusion equation with one dimensional space variable perturbed by a fast process, and proved the weak convergence to a Gaussian process, whose covariance is explicitly described. Later, this was generalized to general stochastic reaction-diffusion equations by Wang and Roberts \cite{WR}. In both papers, the methods of proof are based on Khasminskii's time discretization argument. Recently, we  \cite{RXY} studied the normal deviations for general slow-fast parabolic SPDEs by using the technique of Poisson equation.

\vspace{1mm}

In this paper, we further develop the argument used in \cite{RXY} to study the functional central limit theorem for the  stochastic hyperbolic-parabolic system (\ref{spde01}) with H\"older continuous coefficients. More precisely, we  show that  the normalized difference $Z_t^\eps$, defined by (\ref{zte}), converges weakly as $\eps \to 0$ to the solution of  a linear stochastic wave equation, see {\bf Theorem \ref{main3}}. Moreover, the optimal $1/2$-order rate of convergence is obtained. This rate also   does not depend on the regularity of the coefficients in the equation for the fast variable, which again is natural   since in the limit equation the fast component has been homogenized out. As far as we know, the result we obtained   is completely new. The argument we use to prove the functional central limit theorem is closely and universally connected with the proof of the strong and weak convergence in the averaging principle. We note that due to the model considered in this paper, the framework we deal with is different from \cite{RXY}. Furthermore, we derive the higher order spatial-temporal convergence in the averaging principle and  in the functional central limit theorem. Throughout our proof, several strong and weak fluctuation estimates will play an important role, see Lemmas \ref{strf}, \ref{wef1} and \ref{wfe2} below.

\vspace{1mm}
The rest of this paper is organized as follows. In Section 2, we first introduce some assumptions and state our main results. Section 3 is devoted to establish some preliminary estimates. Then we prove the strong and weak convergence results, Theorem \ref{main1}, and the normal deviation result, Theorem \ref{main3}, in Section 4 and Section 5, respectively.

\vspace{1mm}
{\bf Notations.}
To end this section, we introduce some usual notations for convenience. Given Hilbert spaces $H_1, H_2$ and $\hat H,$ we use $\sL(H_1,H_2)$ to denote the space of all linear and bounded operators from $H_1$ to $H_2$. If $H_1=H_2,$ we write $\sL(H_1)=\sL(H_1,H_1)$  for simplicity. Recall that an operator $Q\in\sL(\hat H)$ is called Hilbert-Schmidt if
$$
\Vert Q\Vert_{\mathscr{L}_2(\hat H)}^2:=Tr(QQ^{*})<+\infty.
$$
We shall denote the space of all Hilbert-Schmidt operators on $\hat H$ by $\mathscr{L}_2(\hat H)$.
Let $L^\infty_{\ell}(H_1\times H_2,\hat H)$ denote the space of  all measurable maps $\phi: H_1\times H_2\to \hat H$ with linear growth, i.e.,
$$
\|\phi\|_{L^\infty_{\ell}(\hat H)}:=\sup_{(x,y)\in H_1\times H_2}\frac{\|\phi(x,y)\|_{\hat H}}{1+\|x\|_{H_1}+\|y\|_{H_2}}<\infty.
$$

For $k\in\mN$, the space $C_{\ell}^{k,0}(H_1\times H_2,\hat H)$ contains all $\phi\in L^\infty_\ell(H_1\times H_2,\hat H)$ such that
$\phi$ has $k$ times G\^ateaux derivatives with respect to the $x$-variable satisfying
$$
\|\phi\|_{C_{\ell}^{k,0}(\hat H)}:=\sup_{(x,y)\in H_1\times H_2}\frac{\sum\limits_{\i=1}^k\|D_x^i\phi(x,y)\|_{\sL^i(H_1,\hat H)}}{1+\|x\|_{H_1}+\|y\|_{H_2}}<\infty.
$$
Similarly, the space $C^{0,k}_{\ell}(H_1\times H_2,\hat H)$ contains all $\phi\in L^\infty_\ell(H_1\times H_2,\hat H)$ such that
$\phi$ has $k$ times G\^ateaux derivatives with respect to the $y$-variable satisfying
\begin{align*}
\|\phi\|_{C_\ell^{0,k}(\hat H)}:=\sup_{(x,y)\in H_1\times H_2}\frac{\sum\limits_{i=1}^k\|D_y^i\phi(x,y)\|_{\sL^i(H_2,\hat H)}}{1+\|x\|_{H_1}+\|y\|_{H_2}}<\infty.
\end{align*}
For $k,m\in\mN$, let  $C^{k,m}_\ell(H_1\times H_2,\hat H)$ be the space of all maps satisfying
\begin{align}\label{norm}
\|\phi\|_{C_\ell^{k,m}(\hat H)}:=\|\phi\|_{L^\infty_\ell(\hat H)}+\|\phi\|_{C_\ell^{k,0}(\hat H)}+\|\phi\|_{C_\ell^{0,m}(\hat H)}<\infty,
\end{align}
and for $\eta\in(0,1)$, the space $C^{k,\eta}_{\ell}(H_1\times H_2,\hat H)$ consists of all $\phi\in C^{k,0}_{\ell}(H_1\times H_2,\hat H)$ satisfying
$$
\|\phi(x,y_1)-\phi(x,y_2)\|_{\hat H}\leq C_0\|y_1-y_2\|_{H_2}^\eta\big(1+\|x\|_{H_1}+\|y_1\|_{H_2}+\|y_2\|_{H_2}\big).
$$
The space $C_b^{k,\eta}(H_1\times H_2,\hat H)$ consists of all $\phi\in C_{\ell}^{k,\eta}(H_1\times H_2,\hat H)$  whose $k$ times   G\^ateaux derivatives with respect to the first variable are bounded, and the space $C_B^{k,\eta}(H_1\times H_2,\hat H)$ consists of all maps in $C_b^{k,\eta}(H_1\times H_2,\hat H)$ which are bounded themselves. We also introduce the space $\mC^{k,k}_l(H_1\times H_2,\hat H)$ consisting of all maps which have $k$ times Fr\'echet derivatives with respect to both the first variable and the second variable and satisfy (\ref{norm}).
 The space $\mC^{k,k}_b(H_1\times H_2,\hat H)$ consists of all $\phi\in\mC^{k,k}_l(H_1\times H_2,\hat H)$ with all derivatives bounded. When $\hat H=\mR$, we will omit the letter $\hat H$ for simplicity.

\section{Assumptions and main results}

Let $H:=L^2(D)$ be the usual space of square integrable functions on a bounded open domain $D$ in $\mR^d$ with scalar product and norm denoted by $\langle\cdot,\cdot\rangle$  and $\Vert\cdot\Vert$, respectively.
Let $A$ be the realization of the Laplacian with Dirichlet boundary conditions in $H$. It is  known that there exists a complete orthonormal basis $\{e_n\}_{n\in\mN}$ of $H$ such that
$$
Ae_n=-\lambda_ne_n,
$$
with $0<\lambda_1\leq\lambda_2\leq\cdots\lambda_n\leq\cdots.$
For $\alpha\in\mR$, let $H^\alpha:=\cD((-A)^{\frac{\alpha}{2}})$ be the Hilbert space endowed with the scalar product
$$
\<x,y\>_\alpha:=\<(-A)^{\frac{\alpha}{2}}x,(-A)^{\frac{\alpha}{2}}y\>
=\sum\limits_{n=1}^{\infty}\lambda_n^{\alpha}\<x,e_n\>\<y,e_n\>,\quad\forall x,y\in H^\alpha,
$$
and norm
$$
\|x\|_\alpha:=\left(\sum\limits_{n=1}^{\infty}\lambda_n^{\alpha}\<x,e_n\>^2\right)^{\frac{1}{2}},\quad\forall
x\in H^\alpha.
$$
Then $A$ can be regarded as an operator from $H^\alpha$ to $H^{\alpha-2}$.
 For the drift coefficients $f$ and $g$ given in system (\ref{spde01}), we  introduce two Nemytskii operators $F, G: H\times H\to H$  by
\begin{align}\label{NG}
F(u,y)(\xi):=f(u(\xi),y(\xi)),\quad G(u,y)(\xi):=g(u(\xi),y(\xi)),\quad\xi\in D.
\end{align}
We remark that these operators are not Fr\'echet differentiable in $H$.

To give precise results, it is convenient to write system (\ref{spde01}) in the following abstract formulation in $H$:
\begin{equation} \label{spde1}
\left\{ \begin{aligned}
&\frac{\partial^2U^{\eps}_t}{\partial t^2}=A U^{\eps}_t +F(U^{\eps}_t, Y^{\eps}_t)+\dot W^1_t,\qquad\qquad t\in(0,T],\\
&\frac{\partial Y^{\eps}_t}{\partial t} =\frac{1}{\eps}A Y^{\eps}_t+\frac{1}{\eps}G(U^{\eps}_t, Y^{\eps}_t)+\frac{1}{\sqrt{\eps}} \dot W_t^2,\quad\,\, t\in(0,T],\\
&U^{\eps}_0=u,\,\,\frac{\partial U^{\eps}_t}{\partial t}\big|_{t=0}=v,\,\,Y^{\eps}_0=y.
\end{aligned} \right.
\end{equation}
For $i=1,2$, we assume that $Q_i$ are nonnegative, symmetric operators with respect to $\{e_{n}\}_{n\in\mN}$, i.e.,
$$Q_ie_{n}=\beta_{i,n}e_{n},\;\;\beta_{i,n}>0,n\in\mN.$$
In addition, we  assume that
\begin{align}\label{trace}
Tr(Q_i)=\sum\limits_{n\in\mN}\beta_{i,n}<+\infty,\;\;i=1,2.
\end{align}
Given $u\in H$, consider the following frozen equation:
\begin{align}\label{froz}
\dif Y_t^u=AY_t^u \dif t+G(u,Y_t^u)\dif t+\dif W_t^2, \quad Y_0^u=y\in H.
\end{align}
Under our assumptions below, the process
$Y_t^u$  admits a unique invariant measure $\mu^u(\dif y)$. Then, the averaged equation for system (\ref{spde1}) is
\begin{equation} \label{spde22}
\left\{ \begin{aligned}
&\frac{\partial^2\bar U_t}{\partial t^2}=A \bar U_t+\bar F(\bar U_t)+\dot W^1_t,\;\;t\in (0,T],\\
&\bar U_0=u,\,\,\frac{\partial \bar U_t}{\partial t}|_{t=0}=v ,
\end{aligned} \right.
\end{equation}
where
\begin{align}\label{df1}
\bar F(u):=\int_HF(u,y)\mu^u(\dif y).
\end{align}

Let $\dot U^{\eps}_t:=\p U^{\eps}_t/\p t$ and $\dot{\bar U}_t:=\p\bar U_t/\p t.$
The following is the first main result of this paper.

\bt\label{main1}
	Let $T>0$, $u\in H^1$ and $v, y\in H$. Assume that $f\in C_b^{2,\eta}(\mR\times \mR,\mR)$ and $g\in C^{2,\eta}_B(\mR\times \mR,\mR)$ with $\eta>0.$  Then we have:
	
	\vspace{1mm}
	\noindent (i) (strong convergence) for any  $q\geq 1$,
	\begin{align}\label{rr1}
	\sup\limits_{t\in[0,T]}\mE\left(\| U^{\eps}_t-\bar U_t\|_1^2+\| \dot U^{\eps}_t-\dot{\bar U}_t\|^2\right)^{q/2}\leq C_1\,\eps^{q/2};
	\end{align}
   (ii) (weak convergence) for any  $ \phi\in \mC_b^3(H)$ and $\tilde\phi\in \mC_b^3(H^{-1})$,
	\begin{align}\label{rr2}
	\sup\limits_{t\in[0,T]}\Big(\big|\mE[\phi(U^{\eps}_t)]-\mE[\phi(\bar U_t)]\big|+\big|\mE[\tilde\phi(\dot U^{\eps}_t)]-\mE[\tilde\phi(\dot{\bar U}_t)]\big|\Big)\leq C_2\,\eps,
	\end{align}
	where  $C_1=C(T,u,v,y)$ and $C_2=C(T,u,v,y,\phi,\tilde\phi)$ are positive constants independent of $\eps$ and $\eta$.
	\et

\br
(i) The  $1/2$-order rate of strong convergence in (\ref{rr1}) and  the $1$-order rate of weak convergence in (\ref{rr2}) should be optimal, which coincides with the SDE case as well as the stochastic reaction-diffusion equation case. Moreover, we obtain that
 both the strong and weak convergence rates do not depend on the regularity of the coefficients in the equation for the fast variable. This coincides with the intuition, since in the limit equation the fast component has been averaged out.

\vspace{1mm}
(ii) Note that the coefficients are assumed to be only $\eta$-H\"older continuous with respect to the fast variable, which is sufficient for us to prove the above convergence  in the averaging principle. However, the pathwise uniqueness of solutions for  system (\ref{spde1}) is not clear under such weak assumptions. In particular,  if the system  is not fully coupled in the sense that the fast motion does not depend on the slow variable (i.e., $g(u,y)=g(y)$ in (\ref{spde01})), then the well-posedness for the fast equation with only H\"older continuous coefficients has been proven in \cite[Theorem 7]{DF} by using the Zvonkin's transformation. This in turn implies the strong well-posedness of the whole system (\ref{spde01}).
\er

Recall that $Z_t^\eps$ is defined by (\ref{zte}). In view of (\ref{spde1}) and (\ref{spde22}), we have
\begin{align*}
\frac{\partial^2 Z^\eps_t}{\partial t^2}&=A Z^\eps_t+\frac{1}{\sqrt{\eps}}\Big[F(U_t^\eps,Y_t^\eps)-\bar F(\bar U_t)\Big]\\
&=A Z^\eps_t+\frac{1}{\sqrt{\eps}}\Big[\bar F(U^\eps_t)-\bar F(\bar U_t)\Big]+\frac{1}{\sqrt{\eps}}\delta F(U_t^\eps,Y_t^\eps),
\end{align*}
where
\begin{align*}
\delta F(u,y):=F(u,y)-\bar F(u).
\end{align*}
To study the  homogenization behavior of $Z_t^\eps$, we consider the following Poisson equation:
\begin{align}\label{poF}
\cL_2(u,y)\Psi(u,y)=-\delta F(u,y),
\end{align}
where $\cL_2$ is the generator of the frozen equation (\ref{froz}) given by
\begin{align}\label{ly2}
\cL_2(u,y)\varphi(y)
&:=\<Ay+G(u,y),D_y\varphi(y)\>\no\\
&\,\,\,\quad+\frac{1}{2}Tr\left(D^2_{y}\varphi(y)Q_2^{\frac{1}{2}}(Q_2^{\frac{1}{2}})^*\right),\quad\forall\varphi\in C_\ell^2(H),
\end{align}
and $u\in H$ is regarded as a parameter.
According to Theorem \ref{PP} below, there exists a unique solution $\Psi$ to equation (\ref{poF}).
Then, the limit process $\bar Z_{t}$ of $Z_t^\eps$ turns out to satisfy the following linear stochastic wave equation:
\begin{equation} \label{spdez}
\left\{ \begin{aligned}
&\frac{\partial^2\bar Z_{t}}{\partial t^2}=A\bar Z_{t}+D_u\bar F(\bar U_t).\bar Z_{t}+\sigma(\bar U_t)\dot W_t,\;\;t\in (0,T],\\
&\bar Z_0=0,\,\,\frac{\partial \bar Z_t}{\partial t}|_{t=0}=0 ,
\end{aligned} \right.
\end{equation}
where   $W_t$ is another cylindrical Wiener process independent of $W_t^1$, and $\sigma$ is a Hilbert-Schmidt operator satisfying
\begin{align*}
\frac{1}{2}\sigma(u)\sigma^*(u)=\overline{\delta F\otimes\Psi}(u):=\int_{H}\big[\delta F(u,y)\otimes\Psi(u,y)\big]\mu^u(\dif y).
\end{align*}

 Let
$\dot Z^{\eps}_t:=\p Z^{\eps}_t/\p t$ and $\dot{\bar Z}_t:=\p\bar Z_t/\p t.$ We have the following result.

\bt[Normal deviation]\label{main3}
Let $T>0,$ $f\in C_b^{2,\eta}(\mR\times \mR,\mR)$ and $g\in C^{2,\eta}_B(\mR\times \mR,\mR)$ with $\eta>0.$  Then for any $u\in H^1$, $v, y\in H$, $\phi\in \mC_b^3(H)$ and $\tilde\phi\in \mC_b^3(H^{-1})$, we have
\begin{align}\label{clt}
\sup_{t\in[0,T]}\Big(\big|\mE[\phi(Z_{t}^{\eps})]-\mE[\phi(\bar Z_{t})]\big|+\big|\mE[\tilde\phi(\dot Z_{t}^{\eps})]-\mE[\tilde\phi(\dot{\bar Z}_{t})]\big|\Big)\leq C_3\,\eps^{\frac{1}{2}},
\end{align}
where  $C_3=C(T,u,\dot u,y,\phi,\tilde\phi)>0$ is a constant independent of $\eps$ and $\eta.$
\et

\br
The $1/2$-order rate of convergence in (\ref{clt}) coincides with the SDE case and  should be optimal. Moreover,
the convergence rate does not depend on the regularity of the coefficients in the equation for the fast variable.
\er

\section{Preliminaries}

\subsection{Poisson equation}
We will rewrite the system (\ref{spde1}) as an abstract evolution equation. To this end, we first introduce some notations. For $\alpha\in\mR$, by $\cH^\alpha:=H^\alpha\times H^{\alpha-1}$ we denote the Hilbert space endowed with  the scalar product
$$
\<u,v\>_{\cH^\alpha}:=\<u_1,v_1\>_\alpha+\<u_2,v_2\>_{\alpha-1},\quad \forall u=(u_1,u_2)^T,v=(v_1,v_2)^T\in \cH^\alpha,
$$
and norm
$$
\interleave u\interleave_\alpha^2:=\|u_1\|_\alpha^2+\|u_2\|_{\alpha-1}^2,\quad\forall u=(u_1,u_2)^T\in \cH^\alpha.
$$
For simplicity, we write $\cH:=H\times H^{-1}.$
Let $\Pi_1$ be the canonical projection from $\cH$ to $H,$ and define
$$
V_t^\eps:=\frac{\dif}{\dif t}U_t^\eps\quad\text{and}\quad X_t^\eps:=(U_t^\eps,V_t^\eps)^T.
$$
Then, the system (\ref{spde1}) can be rewritten as
 \begin{equation} \label{}
\left\{ \begin{aligned}\label{spde11}
&\dif X_t^\eps=\cA X^{\eps}_t+{\cF}(X^{\eps}_t, Y^{\eps}_t)+B\dif W^1_t,\\
&\dif Y^{\eps}_t =\eps^{-1}A Y^{\eps}_t+\eps^{-1}\cG(X^{\eps}_t, Y^{\eps}_t)+\eps^{-1/2}\dif W_t^2,\\
&X^{\eps}_0=x,\,Y^{\eps}_0=y,
\end{aligned} \right.
\end{equation}
where $x:=(u,v)^T$, $\cG(x,y):=G(\Pi_1(x),y)$ and
$$
\cA:=\begin{pmatrix}0&I\\A&0\end{pmatrix},\,\, \cF(x,y):=\begin{pmatrix}0\\F(\Pi_1(x),y)\end{pmatrix},\,\,B\dif W_t^1:=\begin{pmatrix}0\\\dif W_t^1\end{pmatrix},
$$
and $F, G$ are defined by (\ref{NG}). Similarly, concerning  the averaged equation (\ref{spde22}), let
$$
\bar V_t:=\frac{\dif}{\dif t}\bar U_t\quad\text{and}\quad \bar X_t:=( \bar U_t,\bar V_t)^T.
$$
Then we can transfer (\ref{spde22}) into a stochastic evolution equation:
\begin{align}\label{spde20}
\dif \bar{X}_t=\cA\bar{X}_t\dif t+\bar{\cF}(\bar{X}_t)\dif t+B\dif W_t^1,\quad \bar X_0=x=(u,v)^T\in \cH,
\end{align}
where
$$
\bar {\cF}(x):=\begin{pmatrix}0\\\bar F(\Pi_1(x))\end{pmatrix},
$$
and  $\bar F$ is defined by (\ref{df1}).
It is known (see e.g. \cite{BDT}) that $\cA$ generates a strongly continuous group $\{e^{t\cA}\}_{t\geq0}$ which is given by
\begin{align}\label{group}
e^{t\cA}=\begin{pmatrix}C_t&(-A)^{-\frac{1}{2}}S_t\\-(-A)^{\frac{1}{2}}S_t&C_t\end{pmatrix},
\end{align}
where $C_t:=\cos((-A)^{\frac{1}{2}})t)$ and $S_t:=\sin((-A)^{\frac{1}{2}})t).$
For any $x\in\cH$, we have $\interleave e^{\cA t}x\interleave_0\leq\interleave x\interleave_0.$ Moreover,
under the assumptions on $f$ and $g$, one can check that $F\in C_b^{2,\eta}(H\times H, H)$ and $G\in C_B^{2,\eta}(H\times H, H).$ By definition, we further have $\cF\in C_b^{2,\eta}(\cH\times H, \cH^1)$ and $\cG\in C_B^{2,\eta}(\cH\times H, H).$
Furthermore, according to   \cite[Lemma 3.7]{RXY},  we also have that $\bar\cF\in C_b^2(\cH,\cH^1)$.

The Poisson equation will be the crucial tool in our paper. Recall that  $\cL_2(u,y)$ is defined by (\ref{ly2}). If there is no confusion possible, we shall also write
\begin{align}\label{ly22}
\cL_2\varphi(y):=\cL_2(x,y)\varphi(y):=\cL_2(\Pi_1(x),y)\varphi(y),\quad\forall\varphi\in C_{\ell}^2(H).
\end{align}
Consider the following Poisson equation:
\begin{align}\label{pois}
\cL_2(x,y)\psi(x,y)=-\phi(x,y),
\end{align}
where $x\in \cH$ is regarded as a parameter, and $\phi: \cH\times H\rightarrow \hat H$ is measurable.
To be well-defined, it is necessary to make the following ``centering" assumption  on $\phi$:
\begin{align}\label{cen}
\int_{H}\phi(x,y)\mu^x(\dif y)=0,\quad\forall x\in \cH.
\end{align}
The following  result has been proven in \cite[Theorem 3.2]{RXY}.

\bt\label{PP}
Let $\eta>0$ and $k=0,1,2$, and assume  $\cG\in C_B^{k,\eta}(\cH\times H,H)$. Then for every $\phi(\cdot,\cdot)\in C_{\ell}^{k,\eta}(\cH\times H,\hat H)$ satisfying (\ref{cen}),
there exists  a unique solution $\psi(\cdot,\cdot)\in \psi\in C^{k,0}_{\ell}(\cH\times H,\hat H)\cap \mC^{0,2}_{\ell}(\cH\times H,\hat H)$ to equation (\ref{pois}) which is given by
\begin{align*}
\psi(x,y)=\int_0^\infty\mE\big[\phi(x,Y_t^x(y))\big]\dif t,
\end{align*}
where $Y_t^x(y)=Y_t^u(y)$ satisfies the frozen equation (\ref{froz}).
\et

\subsection{Moment estimates}
We   prove the following  estimates for the solution $X_t^\eps$ and $Y_t^\eps$ of system (\ref{spde11}).

\begin{lemma}\label{la41}
	Let $T>0,$  $x\in \cH^1, y\in H,$ and let $(X_t^\eps,Y_t^\eps)$ satisfy
	\begin{equation} \label{spde112}
	\left\{ \begin{aligned}
	&X^{\eps}_t =e^{t\cA}x+\int_0^te^{(t-s)\cA}\cF(X^{\eps}_s, Y^{\eps}_s)\dif s+\int_0^te^{(t-s)\cA}B\dif W^1_s,\\
	& Y^{\eps}_t =e^{\frac{t}{\eps}A}y+\eps^{-1}\int_0^te^{\frac{t-s}{\eps}A}\cG(X^{\eps}_s, Y^{\eps}_s)\dif s+\eps^{-1/2}\int_0^te^{\frac{t-s}{\eps}A}\dif W_s^2.\\
	\end{aligned} \right.
	\end{equation}
	Then for any $q\geq1,$ we have
	\begin{align*}
	\sup\limits_{\eps\in(0,1)}\mE\Big(\sup\limits_{t\in[0,T]}\interleave X^{\eps}_t\interleave_1^{2q}\Big)\leq C_{T,q}\big(1+\interleave x\interleave_1^{2q}+\| y\|^{2q}\big)
	\end{align*}
	and
	\begin{align}\label{msy}
	\sup\limits_{\eps\in(0,1)}\sup\limits_{t\in[0,T]}\mE\| Y^{\eps}_t\|^{2q}+\sup\limits_{\eps\in(0,1)}\mE\left(\int_0^T\|Y_t^\eps\|_1^2\dif t\right)^q \leq C_{T,q}(1+\| y\|^{2q}),
	\end{align}
	where $C_{T,q}>0$ is a constant.
\end{lemma}
\begin{proof}
	Applying It\^o's formula (see e.g. \cite[Section 4.2]{LR}) to $\|Y_t^\eps\|^{2q}$ and taking expectation, we have
	\begin{align*}
	\frac{\dif}{\dif t}\mE\|Y_t^\eps\|^{2q}=\frac{2q}{\eps}\mE\left[\|Y_t^\eps\|^{2q-2}\<AY_t^\eps,Y_t^\eps\>\right]
	&+\frac{2q}{\eps}\mE\left[\|Y_t^\eps\|^{2q-2}\<\cG(X_t^\eps,Y_t^\eps),Y_t^\eps\>\right]\no\\
	&+\Big(\frac{q}{\eps}+\frac{2q(q-1)}{\eps}\Big)Tr(Q_2)\mE\|Y_t^\eps\|^{2q-2}.
	\end{align*}
	It follows from Poincar\'e inequality, Young's inequality and (\ref{trace}) that
	\begin{align*}
	\frac{\dif}{\dif t}\mE\|Y_t^\eps\|^{2q}&\leq-\frac{2q\lambda_1}{\eps}\mE\|Y_t^\eps\|^{2q}
	+\frac{2qC_0}{\eps}\mE\|Y_t^\eps\|^{2q-1}\\
	&\quad+\Big(\frac{q}{\eps}+\frac{2q(q-1)}{\eps}\Big)Tr(Q_2)\mE\|Y_t^\eps\|^{2q-2}\leq -\frac{qC_0}{\eps}\mE\|Y_t^\eps\|^{2q}+\frac{C_0}{\eps}.
	\end{align*}
	Using Gronwall's inequality, we obtain
	\begin{align}\label{msy'}
	\mE\|Y_t^\eps\|^{2q}\leq e^{-\frac{qC_0}{\eps}t}\|y\|^{2q}+\frac{C_0}{\eps}\int_0^te^{-\frac{qC_0}{\eps}(t-s)} \dif s\leq C_0(1+\|y\|^{2q}).
	\end{align}
Furthermore, in view of \cite[Theorem 5.3.5]{CPL}, the process $X_t^\eps=(U_t^\eps,V_t^\eps)^T$ enjoys the following energy equality:
	\begin{align*}
	\interleave X_t^\eps\interleave_1^2=\interleave x\interleave_1^2+2\int_0^t\<V_s^\eps,F(U_s^\eps,Y_s^\eps)\>\dif s+2\int_0^t\<U_t^\eps,\dif W_s^1\>+\int_0^tTrQ_1\dif s.
	\end{align*}
	Then it is easy to check that
	\begin{align}\label{x2q}
	\interleave X_t^\eps\interleave_1^{2q}
	\leq\! C_0\bigg(1+\interleave x\interleave_1^{2q}+\Big|\!\int_0^t\!\<V_s^\eps,F(U_s^\eps,Y_s^\eps)\>\dif s\Big|^q+\Big|\int_0^t\<U_t^\eps,\dif W_s^1\>\Big|^q\bigg).
	\end{align}
	On the one hand, note that
	\begin{align}\label{yf}
	&\mE\sup\limits_{0\leq t\leq T}\Big|\int_0^t\<V_s^\eps,F(U_s^\eps,Y_s^\eps)\>\dif s\Big|^q\no\\
&\leq C_1\mE\Big(\int_0^T\|V_s^\eps\|^2\dif s\Big)^q+C_1\mE\left(\int_0^T(1+\|U_s^\eps\|^2+\|Y_s^\eps\|^{2})\dif s\right)^q\no\\
	&\leq C_1\mE\left(\int_0^T\interleave X_s^\eps\interleave_1^{2q}\dif s\right)+C_1\mE\left(\int_0^T(1+\|Y_s^\eps\|^{2q})\dif s\right).
	\end{align}
	On the other hand, in view of  Burkholder-Davis-Gundy's inequality, we have
	\begin{align}\label{yw}
	\mE\sup\limits_{0\leq t\leq T}\Big|\int_0^t\<U_t^\eps,\dif W_s^1\>\Big|^q&\leq C_2TrQ_1\mE\left(\int_0^T\|U_s^\eps\|^2\dif s\right)^{\frac{q}{2}}\no\\
	&\leq C_2\mE\left(\int_0^T\interleave X_s^\eps\interleave_1^{2q}\dif s\right).
	\end{align}
	Combining (\ref{yf}) and (\ref{yw}) with (\ref{x2q}), we get
	\begin{align*}
	\mE\Big(\sup\limits_{0\leq t\leq T}\interleave X_t^\eps\interleave_1^{2q}\Big)
	&\leq C_3(1+\interleave x\interleave_1^{2q})+C_3\mE\left(\int_0^T\interleave X_s^\eps\interleave_1^{2q}+\|Y_s^\eps\|^{2q}\dif s\right).
	\end{align*}
	Thus, it follows from Gronwall's inequality that
	\begin{align*}
	\mE(\sup\limits_{0\leq t\leq T}\interleave X_t^\eps\interleave_1^{2q})
	&\leq C_4\left(1+\interleave x\interleave_1^{2q}+\int_0^T\mE\|Y_s^\eps\|^{2q}\dif s\right),
	\end{align*}
	which together with (\ref{msy'}) yields
	\begin{align*}
	\mE\Big(\sup\limits_{0\leq t\leq T}\interleave X^{\eps}_t\interleave_1^{2q}\Big)\leq C_5(1+\interleave x\interleave_1^{2q}+\| y\|^{2q}).
	\end{align*}
	In order to prove estimate (\ref{msy}), we deduce that
	\begin{align*}
	\mE\left(\int_0^T\|Y_t^\eps\|_1^2\dif t\right)^q&\leq C_q\left(\int_0^T\big\|e^{\frac{t}{\eps}A}y\big\|_1^2\dif t\right)^q\\ &\quad+C_q\,\mE\left(\int_0^T\Big\|\eps^{-1}\int_0^te^{\frac{t-s}{\eps}A}\cG((X^{\eps}_s, Y^{\eps}_s)\dif s\Big\|_1^2\dif t\right)^q\\
	&\quad+C_q\,\mE\left(\int_0^T\Big\|\eps^{-1/2}\int_0^te^{\frac{t-s}{\eps}A}\dif W_s^2\Big\|_1^2\dif t\right)^q=:\sum_{i=1}^3\sY_i(T,\eps).
	\end{align*}
	For the first term, we have
	\begin{align*}
	\sY_1(T,\eps)&\leq C_6\left(\int_0^{T/\eps}\sum\limits_{k=1}^\infty\lambda_ke^{-2\lambda_kt}
	\<y,e_k\>^2\dif t\right)^q\\&\leq C_6\left(\sum\limits_{k=1}^\infty(1-e^{\frac{-2\lambda_kT}{\eps}})\<y,e_k\>^2\right)^q\leq C_6\|y\|^{2q}.
	\end{align*}
	Note that
	\begin{align*}
	\Big\|\eps^{-1}\int_0^te^{\frac{t-s}{\eps}A}\cG(X^{\eps}_s, Y^{\eps}_s)\dif s\Big\|_1&\leq C_7 \eps^{-1}\int_0^t\Big(\frac{t-s}{\eps}\Big)^{-1/2}e^{-\frac{\lambda_1(t-s)}{2\eps}}\|\cG(X^{\eps}_s, Y^{\eps}_s)\|\dif s\\
	&\leq C_7 \int_0^{t/\eps}\frac{e^{-\frac{\lambda_1s}{2}}}{s^{1/2}} \dif s\leq C_7,
	\end{align*}
	which implies that
	\begin{align*}
	\sY_2(T,\eps)\leq C_8.
	\end{align*}
	For the last term, by Minkowski's inequality, Burkholder-Davis-Gundy's inequality and {(\ref{trace})}, we deduce that
	\begin{align*}
	\sY_3(T,\eps)&\leq C_{9}\bigg\{\int_0^T\left(\mE\Big\|\eps^{-1/2}\int_0^te^{\frac{t-s}{\eps}A}\dif W_s^2\Big\|_1^{2q}\right)^{1/q}\dif t\bigg\}^q\\
	&\leq
	C_{9}\bigg\{\int_0^T\bigg(\mE\Big(\eps^{-1}\int_0^t\sum\limits_{k=1}^
	\infty\lambda_ke^{-2\lambda_k\frac{t-s}{\eps}}\<Q_2e_k,e_k\>\dif s\Big)^q\bigg)^{1/q}\dif t\bigg\}^q\\
	&\leq C_{9}\bigg\{\int_0^T\bigg(\mE\Big(\int_0^{t/\eps}\sum\limits_{k=1}^\infty\lambda_k
	e^{-2\lambda_ks}\<Q_2e_k,e_k\>\dif s\Big)^q\bigg)^{1/q}\dif t\bigg\}^q\leq C_{9}.
	\end{align*}
	Combining the above computations, we get the desired result.
\end{proof}

We also need the following  estimate for $ \cA  X_t^\eps $.

\begin{lemma}\label{la44}
	 Let $T>0,$  $x=(u,v)^T\in \cH^1$ and $y\in H$. Then for any $q\geq1$ and $t\in[0,T],$ we have
	\begin{align*}
	\mE\interleave \cA  X_t^\eps\interleave_0^q\leq C_{T,q}(1+\interleave x\interleave_1^q+\|y\|^{q}),
	\end{align*}
	where $C_{T,q}>0$ is a constant.
\end{lemma}
\begin{proof}
By definition, we have
	$$\cA X_t^\eps=\begin{pmatrix}0&I\\A&0\end{pmatrix}
	\begin{pmatrix}U_t^\eps\\V_t^\eps\end{pmatrix}=
	\begin{pmatrix}V_t^\eps\\AU_t^\eps\end{pmatrix}.$$
	Thus, we deduce that
	\begin{align*}
	\interleave \cA  X_t^\eps\interleave_0^q&\leq C_q\left(\|V_t^\eps\|^q+\|AU_t^\eps\|^q_{-1}\right)\\
	&=C_q\left(\|V_t^\eps\|^q+\|(-A)^{\frac{1}{2}}U_t^\eps\|^q\right).
	\end{align*}
	It then follows from (\ref{spde112}) that
	\begin{align*}
	\mE\|(-A)^{\frac{1}{2}}U_t^\eps\|^q\leq&C_q\big(\|(-A)^{\frac{1}{2}}C_tu\|+\|S_tv\|\big)^q
	+C_q\mE\Big\|\int_0^tS_{t-s}F(U_s^\eps,Y_s^\eps)\dif s\Big\|^q\\
	&+C_q\mE\Big\|\int_0^tS_{t-s}\dif W_s^1\Big\|^q:=\sum_{i=1}^3\sU_i(t,\eps).
	\end{align*}
	For the first term, we have
	$$\sU_1(t,\eps)\leq C_1\interleave x\interleave_1^q.$$
	To control the second term, by Minkowski's inequality and Lemma \ref{la41}, we get
	\begin{align*}
	\sU_2(t,\eps)&\leq C_2\Big(\int_0^t\big(1+\mE\|U_s^\eps\|^q+\mE\|Y_s^\eps\|^{q}\big)^{1/q}\dif s\Big)^q\leq C_2(1+\interleave x\interleave_1^q+\|y\|^{q}).
	\end{align*}
	Finally, by Burkholder-Davis-Gundy's inequality, we obtain
	$$\sU_3(t,\eps)\leq C_3.$$
	Combining the above estimates, we have
	$$\mE\|(-A)^{\frac{1}{2}}U_t^\eps\|^q\leq C_4(1+\interleave x\interleave_1^{q}+\|y\|^{q}).$$
	Note that
	\begin{align*}
	V_t^\eps=-(-A)^{\frac{1}{2}}S_tu+C_tv+\int_0^tC_{t-s}F(U_s^\eps,Y_s^\eps)\dif s+\int_0^tC_{t-s}\dif W_s^1.
	\end{align*}
	In a similar way, we can prove that
	$$\mE\|V_t^\eps\|^q\leq C_5(1+\interleave x\interleave_1^{q}+\|y\|^{q}).$$
Combining the above, we get the desired result.
\end{proof}

The following estimates for the solution of the  averaged equation (\ref{spde20}) can be proved in a similar way as Lemmas \ref{la41} and \ref{la44}, hence we omit the details here.

\begin{lemma}
	Let $T>0$ and $x\in \cH^1$. The averaged equation (\ref{spde20}) admits a unique mild solution $\bar X_t$ such that for all $t\geq0,$
	\begin{align}\label{msb}\bar{X}_t =e^{t\cA}x+\int_0^te^{(t-s)\cA}\bar{\cF}(\bar{X}_s)\dif s+\int_0^te^{(t-s)\cA}B\dif W^1_s.\end{align}
Moreover,  for any $q\geq1$ we have
	\begin{align*}
	\sup\limits_{\eps\in(0,1)}\mE\Big(\sup\limits_{t\in[0,T]}\interleave\bar X_t\interleave_1^{2q}\Big)\leq C_{T,q}(1+\interleave x\interleave_1^{2q})
	\end{align*}
and
	\begin{align*}
	\mE\interleave \cA \bar X_t\interleave_0^q\leq C_{T,q}(1+\interleave x\interleave_1^q),
	\end{align*}
	where $C_{T,q}>0$ is a constant.
\end{lemma}

\section{Strong and weak convergence in the averaging principle}

\subsection{Galerkin approximation}
It\^o's formula will be used frequently below in the proof of the main result. However, due to the persence of unbounded operators in the equation, we can not apply It\^o's formula for SPDE (\ref{spde11}) directly. For this reason, we use the following Galerkin approximation scheme, which reduces the infinite dimensional setting to a finite dimensional one.
For every $n\in\mN,$ let $H^n=span\{e_1,e_2,\cdots,e_n\}.$  Denote the projection of $H$ onto $H^n$ by $P_n$, and set
$$
{\cF}_n(x,y):=\begin{pmatrix}0\\P_nF(\Pi_1(x),y)\end{pmatrix},\quad \cG_n(x,y):=P_nG(\Pi_1(x),y).
$$
It is easy to check that $\cF_n$ and $\cG_n$ satisfy the same conditions as $\cF$ and $\cG$ with bounds which are uniform with respect to $n$. Consider the following finite dimensional system:
\begin{equation}\label{xyzn}
\left\{ \begin{aligned}
&\dif X^{n,\eps}_t =\cA X^{n,\eps}_t\dif t+{\cF}_n(X^{n,\eps}_t, Y^{n,\eps}_t)\dif t+P_n\dif W^1_t,\\
&\dif Y^{n,\eps}_t =\eps^{-1}AY^{n,\eps}_t\dif t+\eps^{-1}\cG_n(X^{n,\eps}_t, Y^{n,\eps}_t)\dif t+\eps^{-1/2} P_n\dif W_t^2,
\end{aligned} \right.
\end{equation}
with initial values $X_0^{n,\eps}=x^n\in H^n\times H^n$ and $Y_0^{n,\eps}=y^n\in H^n$.
The corresponding averaged equation for system (\ref{xyz}) is given by
\begin{align}\label{bxn}
\dif \bar{X}^n_t=\cA\bar{X}^n_t\dif t+\bar{\cF}_n(\bar{X}^n_t)\dif t+P_{n}\dif W_t^1,\quad \bar{X}^n_0=x^{n}\in H^n\times H^n,
\end{align}
where
\begin{align}\label{Fn}
\bar{\cF}_n(x):=\int_{H^n}{\cF}_n(x,y)\mu^{x}_n(\dif y),
\end{align}
and $\mu^{x}_n(\dif y)$ is the invariant measure  associated with the transition
semigroup of the process $Y_t^{x,n}(y)$ which satisfies the frozen equation
\begin{align*}
\dif Y^{x,n}_t =AY^{x,n}_t\dif t+ \cG_n(x^n, Y^{x,n}_t)\dif t+ P_n\dif W_t^2,\;\;Y^{x,n}_0=y^{n}\in H^n.
\end{align*}
Recall that $Y_t^u(y)$ satisfies (\ref{froz}) and note that  $\cG(x,y)=G(u,y).$ We know that $Y_t^{x,n}(y^n)$ converges  strongly to $Y_t^x(y):=Y_t^u(y)$.
Let $T>0,x\in \cH^1$ and $y\in H.$  Then as shown in the proof of \cite[Lemma 3.1]{CPL2}, for any $q\geq 1$ and $t\in[0,T],$ we have
\begin{align*}
\lim\limits_{n\to\infty}\mE\interleave X_t^{\eps}- X^{n,\eps}_t\interleave_1^q=0.
\end{align*}
Furthermore, in view of (\ref{spde112}), (\ref{msb}) and (\ref{group}) we deduce that
\begin{align*}
&\mE\interleave\bar X^n_t-\bar X_t\interleave_1^q\leq \mE\Big|\!\Big|\!\Big| \int_0^te^{(t-s)\cA}(I-P_n)B\dif W_s^1\Big|\!\Big|\!\Big| _1^q\\
&+\mE\!\left(\int_0^t\!\Big(\big\| (-A)^{-\frac{1}{2}}S_{t-s}(\bar F(\bar U_s)-\bar F_n(\bar U_s))\big\|_1\!+\!\big\| C_{t-s}(\bar F(\bar X_s)-\bar F_n(\bar U_s)\big\|\Big)\dif s\right)^q\\
&+\mE\!\left(\int_0^t\!\Big(\big\| (-A)^{-\frac{1}{2}}S_{t-s}(\bar F_n(\bar U_s)-\bar F_n(\bar U_s^n))\big\|_1\!+\!\big\| C_{t-s}(\bar F_n(\bar U_s)-\bar F_n(\bar U_s^n))\big\|\Big)\dif s\right)^q.
\end{align*}
Since $\|\bar F_n-\bar F\|\to0$ as $n\to\infty$ (see e.g. \cite[(4.4)]{Br1}), the first two terms go to 0 as $n\to\infty$ by the dominated convergence theorem. For the last term,  we have
\begin{align*}
&\mE\left(\int_0^t\Big(\big\| (-A)^{-\frac{1}{2}}S_{t-s}\big(\bar F_n(\bar U_s)-\bar F_n(\bar U_s^n))\big\|_1+\big\| C_{t-s}(\bar F_n(\bar U_s)-\bar F_n(\bar U_s^n))\big\|\Big)\dif s\right)^q\\
&\leq  C_1\;\mE\left(\int_0^t\|\bar U_s-\bar U^n_s\|_1\dif s\right)^q\leq  C_1\;\mE\left(\int_0^t\interleave\bar X_s-\bar X^n_s\interleave_1\dif s\right)^q,
\end{align*}
which in turn yields by Gronwall's inequality that
$$
\lim\limits_{n\to\infty}\mE\interleave\bar X^n_t-\bar X_t\interleave_1^q=0.
$$
 Therefore,
in order to prove Theorem {\ref{main1}}, we only need to show that for any $q\geq1,$
\begin{align}\label{nsx}
\sup_{t\in[0,T]}\mE\interleave X_t^{n,\eps}-\bar X^n_t\interleave_1^q\leq C_T\,\eps^{q/2},
\end{align}
and for every $\varphi\in \mC_b^3(\cH)$,
\begin{align}\label{nwx}
\sup_{t\in[0,T]}\big|\mE[\varphi(X_t^{n,\eps})]-\mE[\varphi(\bar X^n_t)]\big|\leq C_T\,\eps,
\end{align}
where {\bf $C_T>0$ is a constant independent of $n$}.
In the rest of this section, we shall only work with the approximating system (\ref{xyzn}), and prove
bounds that are uniform with respect to $n$.
To simplify the notations, we omit the index $n.$ In particular, the space $H^n$ is  denoted by $H$.

\subsection{Proof of Theorem \ref{main1} (strong convergence)}

For simplicity, let
\begin{align}\label{L1}
\cL_1\varphi(x):=\cL_1(x,y)\varphi(x):&=\<\cA x+{\cF}(x,y),D_x\varphi(x)\>_\cH\no\\&\;+\frac{1}{2}Tr\left(D^2_{x}\varphi(x)(BQ_1)^{\frac{1}{2}}((BQ_1)
^{\frac{1}{2}})^*\right),\quad\forall\varphi\in  C_{\ell}^2(\cH).
\end{align}
As shown in Subsection 4.1, to prove the strong convergence result (\ref{rr1}), we only need to prove (\ref{nsx}). To this end, we first establish the following fluctuation estimate for an integral functional of $(X_s^\eps,Y_s^\eps)$ over time interval $[0,t],$  which will play an important role in proving (\ref{nsx}).

\bl[Strong fluctuation estimate]\label{strf}
Let $T,\eta>0,{x=(u,v)^T\in \cH^1}$ and $y\in H.$ Assume that ${\cF}\in C_b^{2,\eta}(\cH\times H,\cH^1)$ and $\cG\in C^{2,\eta}_B(\cH\times H,H).$ Then for any $t\in[0,T]$, $q\geq1$ and every $\tilde\phi(x,y):=\begin{pmatrix}0\\\phi(u,y))\end{pmatrix}$ satisfying (\ref{cen}) with $\phi\in C_b^{2,\eta}(H\times H,H),$  we have
\begin{align*}
\mE\Big|\!\Big|\!\Big|\int_0^t e^{(t-s)\cA}\tilde\phi(X_s^\eps,Y_s^\eps)\dif s\Big|\!\Big|\!\Big|_1^q\leq C_{T,q}\,\eps^{q/2},
\end{align*}
where $C_{T,q}>0$ is a constant independent of $\eps,\eta$ and $n$.
\el

\begin{proof}
Let $\psi$  solve the Poisson equation,
$$
\cL_2(u,y)\psi(u,y)=-\phi(u,y),
$$
and define
$$
\tilde\psi_{t}(s,x,y):= e^{(t-s)\cA}\tilde\psi(x,y):=e^{(t-s)\cA}\begin{pmatrix}0\\\psi(u,y)\end{pmatrix}.
$$
Since $\cL_2$ is an operator with respect to the $y$-variable, one can check that
\begin{align}\label{ppo}
\cL_2\tilde\psi_{t}(s,x,y)=-e^{(t-s)\cA}\tilde\phi(x,y).
\end{align}
Applying It\^o's formula to $\tilde\psi_{t}(t, X_t^\eps,Y_t^{\eps})$, we get
\begin{align}\label{ito1}
\tilde\psi_{t}(t, X_t^\eps,Y_t^{\eps})&=\tilde\psi_{t}(0,x,y)+\int_0^t (\p_s+\cL_1)\tilde\psi_{t}(s,X_s^\eps,Y_s^{\eps})\dif s\no\\
&\quad+\frac{1}{\eps}\int_0^t\cL_2\tilde\psi_{t}(s,X_s^\eps,Y_s^{\eps})\dif s+M_t^1+\frac{1}{\sqrt{\eps}}M_t^2,
\end{align}
where $M_t^1$ and $M_t^2$ are  defined by
\begin{align*}
M_t^1:=\int_0^t D_x\tilde\psi_{t}(s,X_s^\eps,Y_s^{\eps})B\dif W_s^1\quad\text{and}\quad M_t^2:=\int_0^t D_y\tilde\psi_{t}(s,X_s^\eps,Y_s^{\eps})\dif W_s^2.
\end{align*}
Multiplying both sides of (\ref{ito1}) by $\eps$  and using (\ref{ppo}), we obtain
\begin{align}\label{poi}
&\int_0^t e^{(t-s)\cA}\tilde\phi(X_s^\eps,Y_s^{\eps})\dif s=-\int_0^t\cL_2\tilde\psi_{t}(s,X_s^\eps,Y_s^{\eps})\dif s\no\\
&=\eps\big[\tilde\psi_{t}(0,x,y)-\tilde\psi_{t}(t,X^\eps_t,Y_t^{\eps})\big]
+\eps\int_0^t\p_s\tilde\psi_{t}(s,X_s^\eps,Y_s^{\eps})\dif s\no\\
&\quad+\eps\int_0^t\cL_1\tilde\psi_{t}(s,X_s^\eps,Y_s^{\eps})\dif s+\eps M_t^1+\sqrt{\eps}M_t^2=:\sum_{i=1}^5\sJ_i(t,\eps).
\end{align}
According to Theorem \ref{PP} , we have that $\psi \in C_{\ell}^{2,2}(H\times H,H)$ and hence
\begin{align*}
\interleave e^{(t-s)\cA}\tilde\psi(x,y)\interleave_1&=\Big|\!\Big|\!\Big| \begin{pmatrix}(-A)^{-\frac{1}{2}}S_{t-s}\psi(u
,y)\\C_{t-s}\psi(u,y)\end{pmatrix}\Big|\!\Big|\!\Big|_1
\\&=\| (-A)^{-\frac{1}{2}}S_{t-s}\psi(u,y)\|_1+\| C_{t-s}\psi(u,y)\|\\&\leq2\|\psi(u,y)\|\leq C_1(1+\| u\|+\|y\|).
&\end{align*}
As a result,   by Lemma {\ref{la41}} we get
 \begin{align*}
\mE\interleave\sJ_{1}(t,\eps)\interleave_1^q&\leq C_1\,\eps^q(1+\mE\| U_t^\eps\|^q+\mE\|Y_t^{\eps}\|^{q})\leq C_1\,\eps^q.
\end{align*}
Note that
$$
\p_s\tilde\psi_{t}(s,x,y)=-\cA e^{(t-s)\cA}\tilde\psi(x,y),
$$
and that
\begin{align*}
\interleave \cA e^{(t-s)\cA}\tilde\psi(x,y)\interleave_1&=\Big|\!\Big|\!\Big| \begin{pmatrix}C_{t-s}\psi(u,y)\\-(-\cA)^{\frac{1}{2}}S_{t-s}\psi(u,y)\end{pmatrix}\Big|\!\Big|\!\Big|_1
\\&=\| C_{t-s}\psi(u,y)\|_1+\| -(-A)^{\frac{1}{2}}S_{t-s}\psi(u,y)\|\\&\leq2\|\psi(u,y)\|_1\leq C_2(1+\| u\|_1^2+\|y\|_1^2),
\end{align*}
where the last inequality can be obtained as in \cite[(2.16)]{CG}. Thus, using Minkowski's inequality and Lemma \ref{la41} again, we have
\begin{align*}
\mE\interleave\sJ_{2}(t,\eps)\interleave_1^q&\leq C_2\,\eps^q\,\Big(\int_0^T\big(1+\mE\|U_s^\eps\|_1^{2q}\big)^{1/q}\dif t\Big)^q\\
&\quad+C_2\,\eps^q\,\mE\Big(\int_0^T\|Y_s^{\eps}\|_1^2\dif s\Big)^q\leq C_2\,\eps^q.
\end{align*}
For the third term, we have
\begin{align*}
&|\cL_1\tilde\psi_{t}(s,X_s^\eps,Y_s^{\eps})|\leq C_3\,\big(1+\interleave \cA X_s^\eps\interleave_0^2+\interleave X_s^\eps\interleave_1^2+\|Y_s^{\eps}\|^2\big),
\end{align*}
which together with Minkowski's inequality, Lemmas \ref{la41} and {\ref{la44}} yields that
\begin{align*}
\mE\!\interleave\!\sJ_{3}(t,\eps)\interleave_1^q&\leq\! C_3\,\eps^q\bigg(\!\int_0^t\!\Big(\mE\big(1\!+\!\interleave \cA X_s^\eps\interleave_0^2+\interleave X_s^\eps\interleave_1^2+\| Y_s^{\eps}\|^2\big)^q\Big)^{1/q}\dif s\!\bigg)^q\!\leq\! C_3\,\eps^q.
\end{align*}
Finally, by Burkholder-Davis-Gundy's inequality, Theorem {\ref{PP}}, Lemma \ref{la41} and (\ref{trace}), we have
\begin{align*}
\mE\interleave\sJ_{4}(t,\eps)\interleave_1^q&\leq C_4\,\eps^q\Big(\int_0^T\mE\big\| e^{(t-s)\cA}D_x\tilde\psi(X_s^\eps,Y_s^{\eps})BQ_1^{\frac{1}{2}}\big\|^{2}_{\sL_2(\cH^1)}\dif  s\Big)^{q/2}\\
&\leq  C_4\,\eps^q\left(\int_0^T(1+\mE\interleave X_s^\eps\interleave_1^{2}+\mE\|Y_s^{\eps}\|^{2})\dif  s\right)^{q/2}\leq  C_4\,\eps^q,
\end{align*}
and similarly,
\begin{align*}
\mE\interleave\sJ_{5}(t,\eps)\interleave_1^q \leq  C_5\,\eps^{q/2}.
\end{align*}
Combining the above inequalities with (\ref{poi}), we get
the desired estimate.
\end{proof}

We are now in the position to give:

\begin{proof}[Proof of estimate (\ref{nsx})]
	
Fix $T>0$ below. In view of (\ref{spde112}) and (\ref{msb}), for every $t\in[0,T]$ we have
\begin{align*}
X_t^\eps-\bar X_t&=\int_0^t e^{(t-s)\cA}\big[\bar \cF(X_s^\eps)-\bar \cF(\bar X_s)\big]\dif s+\int_0^t e^{(t-s)\cA}\delta \cF(X_s^\eps,Y_s^{\eps})\dif s,
\end{align*}
where $\delta \cF$ is defined by
\begin{align}\label{dF}
\delta \cF(x,y):=\cF(x,y)-\bar \cF(x)=\begin{pmatrix}0\\ \delta F(\Pi_1(x),y)\end{pmatrix}.
\end{align}
Thus, we have for any $q\geq1,$
\begin{align*}
\mE \interleave X_t^\eps-\bar X_t\interleave_1^q&\leq C_0\, \mE\Big|\!\Big|\!\Big|\int_0^t e^{(t-s)\cA}\big[\bar \cF(X_s^\eps)-\bar \cF(\bar X_s)\big]\dif s\Big|\!\Big|\!\Big|_1^q\\
&+C_0\,\mE\Big|\!\Big|\!\Big|\int_0^t e^{(t-s)\cA}\delta \cF(X_s^\eps,Y_s^{\eps})\dif s\Big|\!\Big|\!\Big|_1^q=:\sI_1(t,\eps)+\sI_2(t,\eps).
\end{align*}
Since $\bar\cF\in C_b^2(\cH,\cH^1),$ by Minkowski's inequality we deduce that
\begin{align*}
\sI_1(t,\eps)&\leq C_1\mE\Big(\int_0^t\interleave\bar \cF(X_s^\eps)-\bar \cF(\bar X_s)\interleave_1\dif s\Big)^q\leq C_1\int_0^t\mE\interleave X_s^\eps-\bar X_s\interleave_1^q\dif s.
\end{align*}
For the second term, noting that $\delta \cF(x,y)$ satisfies the centering condition (\ref{cen}),
it follows by Lemma \ref{strf} directly that
\begin{align*}
\sI_{2}(t,\eps)\leq C_2 \,\eps^{q/2}.
\end{align*}
Thus, we arrive at
\begin{align*}\label{}
\mE \interleave X_t^\eps-\bar X_t\interleave_1^q\leq C_3\,\eps^{q/2}+C_3\,\int_0^t\mE\interleave X_s^\eps-\bar X_s\interleave_1^q\dif s,
\end{align*}
which together with Gronwall's inequality yields the desired result.
\end{proof}

\subsection{Proof of Theorem \ref{main1} (weak convergence)}
As in the previous subsection, to prove the weak convergence result in Theorem \ref{main1} , we only need to show (\ref{nwx}).
The main reason for the difference between the strong and weak convergence rates in the averaging principle can be seen through the following  estimate.

\bl[Weak fluctuation estimate]\label{wef1}
Let $T,\eta>0,{x=(u,v)^T\in \cH^1}$ and $y\in H.$ Assume that ${\cF}\in C_b^{2,\eta}(\cH\times H,\cH^1)$ and $\cG\in C^{2,\eta}_B(\cH\times H,H).$ Then for any $t\in[0,T]$, $\phi\in C_{\ell}^{1,2,\eta}([0,T]\times\cH\times H)$ satisfying (\ref{cen}) and
\begin{align}\label{as01}
|\p_t\phi(t,x,y)|\leq C_0(1+\interleave x\interleave_1^2+\|y\|^2),
\end{align}
we have
\begin{align*}
\mE\left(\int_0^t\phi(s,X_s^\eps,Y_s^\eps)\dif s\right)\leq C_T\;\eps,
\end{align*}
where $C_T>0$ is a constant independent of $\eps,\eta$  and $n$.
\el
\begin{proof}

Let $\psi$  solve the Poisson equation
\begin{align}\label{pot}
\cL_2\psi(t,x,y)=-\phi(t,x,y),
\end{align}
where $\cL_2$ is given by (\ref{ly22}).
According to Theorem \ref{PP}, we can apply  It\^o's formula to $\psi(t,X_t^\eps,Y_t^{\eps})$ to get that
\begin{align*}
\mE[\psi(t,X_t^\eps,Y_t^{\eps})]&=\psi(0,x,y)+
\mE\left(\int_0^t(\partial_s+\mathcal{L}_1)\psi(s,X_s^\eps,Y_s^{\eps})\dif s\right)\\
&\quad+\frac{1}{\eps}\mE\left(\int_0^t\mathcal{L}_2\psi(s,X_s^\eps,Y_s^{\eps})\dif s\right).
\end{align*}
Combining this with (\ref{pot}), we obtain
\begin{align*}
&\mE\left(\int_0^t\phi(s,X_s^\eps,Y_s^{\eps})\dif s\right)\\
&=\eps\mE\big[\psi(0,x,y)-\psi(t,X_t^\eps,Y_t^{\eps})\big]
+\eps\mE\left(\int_0^t\mathcal{L}_1\psi(s,X_s^\eps,Y_s^{\eps})\dif s\right)\\
&\quad+\eps\mE\left(\int_0^t\partial_s\psi(s,X_s^\eps,Y_s^{\eps})\dif s\right)=:\sum_{i=1}^3\sW_i(t,\eps).
\end{align*}
By using exactly the same arguments as in the proof of Lemma \ref{strf}, we can get that
\begin{align*}
\sW_1(t,\eps)+\sW_2(t,\eps)\leq C_1\,\eps.
\end{align*}
To control the third term, note that
$$
\cL_2\p_t\psi(t,x,y)=-\p_t\phi(t,x,y).
$$
In view of  condition (\ref{as01}), we have
\begin{align*}
|\p_t\psi(t,x,y)|\leq C_0(1+\interleave x\interleave_1^2+\|y\|^{2}),
\end{align*}
which together with Lemma \ref{la41} implies that
\begin{align*}
\sW_3(t,\eps)&\leq C_2\,\eps\mE\left(\int_0^t(1+\interleave X_s^\eps\interleave_1^2+\|Y_s^{\eps}\|^{2})\dif s\right)\leq C_2\,\eps.
\end{align*}
Combining the above estimates, we get the desired result.
\end{proof}

Given $T>0$,  consider the  following Cauchy problem on $[0,T]\times\cH$:
\begin{equation} \label{ke}
\left\{ \begin{aligned}
&\partial_t\bar u(t,x)=\bar \cL_1\bar u(t,x),\quad t\in(0,T],\\
& \bar u(0,x)=\varphi(x),
\end{aligned} \right.
\end{equation}
where $\varphi:\cH\to\mR$ is measurable and $\bar\cL_1$ is formally the infinitesimal generator of the process $\bar X_t$ given by
\begin{align}\label{lxb}
\bar\cL_1\varphi(x)&=\<\cA x+\bar \cF(x),D_x\varphi(x)\>_\cH\no\\
&\quad+\frac{1}{2}Tr\left(D^2_{x}\varphi(x)(BQ_1)^{\frac{1}{2}}((BQ_1)^{\frac{1}{2}})
^*\right),\quad\forall\varphi\in C_{\ell}^2(\cH).
\end{align}
 The following result has been proven in \cite[Lemmas A.3-A.5 and 4.3]{FWLL}.
\bl\label{th47}
For every $\varphi\in \mC_b^3(\cH)$, there exists a solution $\bar u\in C_b^{1,3}([0,T]\times \cH)$ to equation (\ref{ke}) which is given by
\begin{align*}
\bar u(t,x)=\mE\big[\varphi(\bar X_t(x))\big].
\end{align*}
Moreover,
for any $t\in[0,T]$ and $x, h\in\cH^1$, we have
\begin{align*}
\vert \p_tD_x\bar u(t,x).h|\leq C_T\interleave h\interleave_1(1+\interleave x\interleave_1),
\end{align*}
where $C_T>0$ is a constant.
\el

Now, we are  in the position to give:

\begin{proof}[Proof of estimate (\ref{nwx})]
Given $T>0$ and  $\varphi\in \mC_b^3(\cH)$, let $\bar u$ solve the Cauchy problem (\ref{ke}). For any $t\in[0,T]$ and $x\in \cH^1$, define
$$
\tilde u(t,x):=\bar u(T-t,x).
$$
Then one can check that
$$
\tilde u(T,x)=\bar u(0,x)=\varphi(x)\quad\text{and}\quad\tilde u(0,x)=\bar u(T,x)=\mE[\varphi(\bar X_T(x))].
$$
Using It\^o's formula and taking expectation, we deduce that
\begin{align*}
\mE[\varphi(X_T^{\eps})]-\mE[\varphi(\bar X_T)]&=\mE[\tilde u(T,X_T^{\eps})-\tilde u(0,x)]\\
&=\mE\left(\int_0^T\big(\p_t+\cL_1\big)\tilde u(t,X_t^{\eps})\dif t\right)\\
&=\mE\left(\int_0^T[\mathcal{L}_1\tilde u(t,X_t^{\eps})-\mathcal{\bar L}_1\tilde u(t,X_t^{\eps})]\dif t\right)\\
&=\mE\left(\int_0^T\langle \delta \cF(X_t^\eps,Y_t^\eps), D_x\tilde u(t,X_t^{\eps})\rangle_\cH \dif t\right).
\end{align*}
Note that the function
$$
\phi(t,x,y):=\<\delta \cF(x,y), D_x\tilde u(t,x)\>_\cH
$$
satisfies the centering condition (\ref{cen}).
Moreover, by Lemma \ref{th47} we have
\begin{align*}
\p_t\phi(t,x,y)=\<\delta \cF(x,y), \p_tD_x\bar u(T-t,x)\>_\cH
\leq C_0(1+\interleave x\interleave_1^2+\|y\|^2).
\end{align*}
As a result of Lemma \ref{wef1}, we have
\begin{align*}
\mE[\varphi(X_T^{\eps})]-\mE[\varphi(\bar X_T)]\leq C_1\,\eps,
\end{align*}
which completes the proof.
\end{proof}

\section{Normal deviations}

\subsection{Cauchy problem}
Define
\begin{align*}
\cZ_t^\eps:=\frac{X_t^\eps-\bar X_t}{\sqrt{\eps}}.
\end{align*}
 In view of (\ref{spde11}) and (\ref{spde20}), we consider the process $(X^{\eps}_t, Y^{\eps}_t, \bar X_t, \cZ_t^{\eps})$ as the solution to the following system of equations:
\begin{equation}\label{xyz}
\left\{ \begin{aligned}
&\dif X^{\eps}_t =\!\cA X^{\eps}_t\dif t+\cF(X^{\eps}_t, Y^{\eps}_t)\dif t+B\dif W^1_t,\qquad\qquad\qquad\qquad\quad\quad\,\,\, X_0^\eps=x,\\
&\dif Y^{\eps}_t =\!\eps^{-1}AY^{\eps}_t\dif t+\eps^{-1}\cG(X^{\eps}_t, Y^{\eps}_t)\dif t+\eps^{-1/2} \dif W_t^2,\qquad\qquad\quad
\quad\,\,Y^{\eps}_0=y,\\
&\dif \bar X_t=\!\cA\bar X_t\dif t+\bar \cF(\bar X_t)\dif t+B\dif W^1_t,\quad\qquad\quad\qquad\quad\qquad\quad\quad\,\,\,\,\quad\bar X_0=x,\\
&\dif \cZ_t^{\eps}=\!\cA \cZ_t^\eps \dif t+\eps^{-1/2}[\bar \cF(X_t^\eps)-\bar \cF( \bar X_t)]\dif t+\eps^{-1/2}\delta \cF(X_t^\eps,Y_t^\eps)\dif t, \,\, Z_0^{\eps}=0,
\end{aligned} \right.
\end{equation}
where $\delta \cF$ is defined by (\ref{dF}).
As a result of Theorem \ref{main1}, we have that for any $q\geq1,$
\begin{align}\label{zne}
\sup\limits_{0\leq t\leq T}\mE\interleave \cZ_t^\eps\interleave_1^q\leq C_T<\infty.
\end{align}
Furthermore, note that
\begin{align*}
\cA \cZ_t^\eps=\begin{pmatrix}0&I\\A&0\end{pmatrix}
\begin{pmatrix}\frac{U_t^\eps-\bar U_t}{\sqrt{\eps}}\\\frac{V_t^\eps-\bar V_t}{\sqrt{\eps}}\end{pmatrix}=\begin{pmatrix}\frac{V_t^\eps-\bar V_t}{\sqrt{\eps}}\\\frac{A(U_t^\eps-\bar U_t)}{\sqrt{\eps}}\end{pmatrix},
\end{align*}
hence we have
\begin{align}\label{azne}
\mE\interleave \cA \cZ_t^\eps\interleave_0^q&=\mE\left(\Big\|\frac{V_t^\eps-\bar V_t}{\sqrt{\eps}}\Big\|^2+\Big\|\frac{A(U_t^\eps-\bar U_t)}{\sqrt{\eps}}\Big\|_{-1}^2\right)^{q/2}\no\\
&=\mE\left(\Big\|\frac{V_t^\eps-\bar V_t}{\sqrt{\eps}}\Big\|^2+\Big\|\frac{(U_t^\eps-\bar U_t)}{\sqrt{\eps}}\Big\|_{1}^2\right)^{q/2}
\leq C_T<\infty.
\end{align}

Similarly, we
rewrite (\ref{spdez}) as
\begin{align}\label{zz0}
\dif \bar \cZ_t=\cA\bar \cZ_t\dif t+D_x\bar \cF(\bar X_t).\bar \cZ_t\dif t+\Sigma(\bar X_t)\dif  W_t,
\end{align}
where $\bar \cZ_t=(\bar Z_{t},\dot{\bar Z}_{t})^T$, and $\Sigma$ is a Hilbert-Schmidt operator satisfying
\begin{align}\label{sst1}
\frac{1}{2}\Sigma(x)\Sigma^*(x)=\overline{\delta \cF\otimes\tilde\Psi}(x):=\int_{H}\big[\delta \cF(x,y)\otimes\tilde\Psi(x,y)\big]\mu^x(\dif y),
\end{align}
(see e.g. \cite[(1.6)]{Ce2} and \cite[(11)]{WR}),
and $\tilde\Psi$ is the solution of the following Poisson equation:
\begin{align}\label{poF1}
\cL_2(x,y)\tilde\Psi(x,y)=-\delta \cF(x,y).
\end{align}

Recall that $\cL_2(x,y)=\cL_2(u,y)$ and $\Psi(u,y)$ solves the Poisson equation (\ref{poF}). Thus, we have $\tilde\Psi(x,y)=\Psi(\Pi_1(x),y)=\Psi(u,y).$ Combining (\ref{spde20}) and (\ref{zz0}), the process $(\bar X_t, \bar \cZ_t)$ solves the system
\begin{equation*}
\left\{ \begin{aligned}
&\dif \bar{X}_t=\cA\bar{X}_t\dif t+\bar{\cF}(\bar{X}_t)\dif t+\dif W_t^1,\qquad\qquad\qquad\,\,\, \bar X_0=x,\\
&\dif \bar \cZ_t=\cA\bar \cZ_t\dif t+D_x\bar \cF(\bar X_t).\bar \cZ_t\dif t+\Sigma(\bar X_t)\dif W_t,\qquad\!\!\bar \cZ_0=0.
\end{aligned} \right.
\end{equation*}
Note that the processes $\bar X_t$ and $\bar \cZ_{t}$ depend on the initial value $x$. Below, we shall write $\bar X_t(x)$  when we want to stress its dependence on the
initial value, and use $\bar \cZ_{t}(x,z)$ to denote the process $\bar \cZ_{t}$ with initial point $\bar \cZ_0=z\in \cH$.

Given $T>0,$ consider the following  Cauchy problem on $[0,T]\times\cH\times\cH$:
\begin{equation} \label{kez}
\left\{ \begin{aligned}
&\partial_t\bar u(t,x,z)=\bar \cL\bar u(t,x,z),\quad t\in(0,T],\\
& \bar u(0,x,z)=\varphi(z),\\
\end{aligned} \right.
\end{equation}
where $\varphi:\cH\to\mR$ is measurable and $\bar \cL$ is formally the infinitesimal generator of the  Markov process $(\bar X_t, \bar \cZ_t)$, i.e.,
$$
\bar \cL:=\bar \cL_1+\bar \cL_3,
 $$
 with $\bar \cL_1$ given by (\ref{lxb}) and $\bar \cL_3$  defined by
\begin{align*}
\bar \cL_3\varphi(z):=\bar \cL_3(x,z)\varphi(z)&:=\<\cA z+D_x\bar \cF(x).z,D_z\varphi(z)\>_{\cH}\\&\;+\frac{1}{2}\,Tr\big(D^2_{z}\varphi(z)\Sigma(x)\Sigma^*(x)\big),\quad\forall\varphi\in C_{\ell}^2(\cH).
\end{align*}
We have the following result.

\begin{lemma}\label{bure}
For every $\varphi\in \mC_b^3(\cH)$, there exists a solution $\bar u\in C_b^{1,3,3}([0,T]\times \cH\times \cH)$ to equation (\ref{kez}) which is given by
\begin{align}\label{kol}
\bar u(t,x,z)=\mE\big[\varphi(\bar \cZ_t(x,z))\big].
\end{align}
Moreover,
	for any $t\in[0,T]$ and $x,z,h\in \cH^1$, we have
	\begin{align}\label{utzx}
	\vert \p_tD_z&\bar u(t,x,z).h|+\vert \p_tD_x\bar u(t,x,z).h|\no\\&\leq C_0\big(1+\interleave x\interleave^2_1+\interleave z\interleave_1+\interleave \cA x\interleave_0+\interleave \cA z\interleave_0\big)\big(\interleave h\interleave_1+\interleave \cA h\interleave_0\big),
	\end{align}
where $C_0>0$ is a positive constant.
\end{lemma}
\begin{proof}By using the same arguments as in \cite[Section  7]{Br2},  we can prove that equation (\ref{kez}) admits a solution $\bar u\in C_b^{1,3,3}([0,T]\times \cH\times \cH)$ which is given by (\ref{kol}), see also \cite[Section 4]{Br4}.
 Moreover, for $x,z,h\in \cH^1,$
\begin{align}\label{ptdz}
\partial_tD_z\bar u(t,x,z).h=D_z\partial_t\bar u(t,x,z).h=D_z(\bar\cL_1+\bar\cL_3)\bar u(t,x,z).h,
\end{align}
On the one hand, we have
\begin{align*}
&D_z\bar\cL_1\bar u(t,x,z).h\no\\&=D_zD_x\bar u(t,x,z).(\cA x+\bar \cF(x),h)+\frac{1}{2}\sum\limits_{n=1}^\infty \beta_{1,n}D_zD_x^2\bar u(t,x,z).(Be_{n},Be_{n},h),
\end{align*}
which together with $\bar u\in C_b^{1,3,3}([0,T]\times \cH\times \cH)$ yields that
\begin{align}\label{lxdz}
&|D_z\bar\cL_1\bar u(t,x,z).h|\leq C_1(1+\interleave\cA x\interleave_0+\interleave x\interleave_1)\interleave h\interleave_1.
\end{align}
On the other hand, we have
\begin{align*}
&D_z\bar\cL_3\bar u(t,x,z).h\\&=\<\cA h,D_z\bar u(t,x,z)\>_\cH+\<D_x\bar \cF( x).h,D_z\bar u(t,x,z))\>_\cH\no\\&+D_z^2\bar u(t,x,z).(\cA z+D_x\bar \cF(x).z,h)+\frac{1}{2}\sum\limits_{n=1}^{\infty} D_z^3\bar u(t,x,z).(\Sigma(x) e_{n},\Sigma(x) e_{n},h).
\end{align*}
Thus,
\begin{align}\label{lzdz}
&|D_z\bar\cL_3\bar u(t,x,z).h|\leq C_2(1+\interleave\cA z\interleave_0+\interleave z\interleave_1+\interleave x\interleave_1^2)(\interleave h\interleave_1+\interleave \cA h\interleave_0).
\end{align}
Combining (\ref{ptdz}), (\ref{lxdz}) and (\ref{lzdz}), we arrive at
\begin{align*}
	\vert \p_tD_z\bar u(t,x,z).h|\leq C_3\big(1&+\interleave x\interleave^2_1+\interleave z\interleave_1\\
	&+\interleave \cA x\interleave_0+\interleave \cA z\interleave_0\!\big)\!\big(\!\interleave h\interleave_1+\interleave \cA h\interleave_0\!\big).
	\end{align*}
Similarly, we have
\begin{align*}
&\partial_tD_x\bar u(t,x,z).h=D_x^2\bar u(t,x,z).(\cA x+\bar \cF(x),h)+\<\cA h+D_x\bar \cF(x).h,D_x\bar u(t,x,z)\>_{\cH}\\
&\qquad+\<D^2_x\bar \cF( x).(z,h),D_z\bar u(t,x,z))\>_\cH+D_xD_z\bar u(t,x,z).(\cA z+D_x\bar \cF(x).z,h)\\
&\qquad+\frac{1}{2}\sum\limits_{n=1}^\infty \beta_{1,n}D_x^3\bar u(t,x,z).(Be_{n},Be_{n},h)\\
&\qquad+\frac{1}{2}\sum\limits_{n=1}^{\infty} D_xD_z^2\bar u(t,x,z).(\Sigma(x) e_{n},\Sigma(x) e_{n},h)\\
&\qquad+\sum\limits_{n=1}^{\infty} D_z^2\bar u(t,x,z).(D_x(\Sigma(x)) e_{n},\Sigma(x) e_{n},h).
\end{align*}
By the same argument as above, we can obtain
\begin{align*}
	\vert \p_tD_x\bar u(t,x,z).h|\leq C_4\big(1&+\interleave x\interleave^2_1+\interleave z\interleave_1\\
	&+\interleave \cA x\interleave_0+\interleave \cA z\interleave_0\!\big)\!\big(\!\interleave h\interleave_1+\interleave \cA h\interleave_0\!\big),
	\end{align*}
which completes the proof.
\end{proof}

\subsection{Proof of Theorem \ref{main3}}
As before, we reduce the infinite dimensional problem to a finite dimensional one by the Galerkin approximation. Recall that $X_t^{n,\eps}$ and $\bar X_t^{n}$ are defined by (\ref{xyzn}) and (\ref{bxn}), respectively.
Define
$$
\cZ_t^{n,\eps}:=\frac{X_t^{n,\eps}-\bar X_t^{n}}{\sqrt{\eps}}.
$$
Then we have
\begin{align*}
&\dif \cZ_t^{n,\eps}=\cA \cZ_t^{n,\eps} \dif t+\eps^{-1/2}[\bar \cF_n(X_t^{n,\eps})-\bar \cF_n( \bar X_t^n)]\dif t+\eps^{-1/2}\delta \cF_n(X_t^{n,\eps},Y_t^{n,\eps})\dif t,
\end{align*}
where $\bar \cF_n$ is given by (\ref{Fn}), and $\delta \cF_n(x,y):=\cF_n(x,y)-\bar \cF_n(x)$.
Let $ \bar \cZ_t^n$ satisfy the following linear equation:
\begin{align*}
\dif \bar \cZ_t^n=\cA\bar \cZ_t^n\dif t+D_x\bar \cF_n(\bar X_t^n).\bar \cZ_t^n\dif t+P_n\Sigma(\bar X_t^n)\dif  W_t,
\end{align*}
where $ W_t$ is a cylindrical Wiener process in $H$, and $\Sigma(x)$ is defined by (\ref{sst1}).
As in \cite[Lemma 5.4]{RXY}, one can check that
\begin{align}\label{znz}
\lim\limits_{n\to\infty}\mE\Big(\interleave\cZ^\eps_t-\cZ_t^{n,\eps}\interleave_1+\interleave\bar \cZ_t-\bar \cZ_t^{n}\interleave_1\Big)=0.
\end{align}

For any $T>0$ and $\varphi\in \mC_b^3(\cH),$ we have for $t\in[0,T]$,
\begin{align}\label{wcz}
\left|\mE[\varphi(\cZ_t^{\eps})]-\mE[\varphi(\bar \cZ_t)]\right|&\leq \left|\mE[\varphi(\cZ_t^{\eps})]-\mE[\varphi(\cZ_t^{n,\eps})]\right|\no\\
&\quad+\left|\mE[\varphi(\cZ_t^{n,\eps})]-\mE[\varphi(\bar \cZ_t^n)]\right|+\left|\mE[\varphi(\bar \cZ_t^n)]-\mE[\varphi(\bar \cZ_t)]\right|.
\end{align}
  According to (\ref{znz}), the first and the last terms on the right-hand of (\ref{wcz}) converge to $0$ as $n\to\infty$ . Therefore,
in order to prove Theorem {\ref{main3}}, we only need to show that
\begin{align}\label{nzz}
\sup_{t\in[0,T]}\left|\mE[\varphi(\cZ_t^{n,\eps})]-\mE[\varphi(\bar \cZ_t^n)]\right|\leq C_T\,\eps^{\frac{1}{2}},
\end{align}
where {\bf $C_T>0$ is a constant independent of} $n.$ We shall only work with the approximating system in the following subsection, and proceed to prove
bounds that are uniform with respect to $n$. To simplify the notations, we shall omit the index $n$ as before.

Define
\begin{align}\label{333}
&\cL_3\varphi(z):=\cL_3(x,y,\bar x,z)\varphi(z):=\<\cA z,D_z\varphi(z)\>_{\cH}\\
&\,\,+\frac{1}{\sqrt{\eps}}\<\bar \cF(x)-\bar \cF(\bar x), D_z\varphi(z)\>_{\cH}+\frac{1}{\sqrt{\eps}}\<\delta \cF(x,y),D_z\varphi(z)\>_{\cH},\quad\forall\varphi\in C_{\ell}^1(\cH).\no
\end{align}
Given a function $\phi\in C_{\ell}^{1,2,\eta,2}([0,T]\times\cH\times H\times \cH)$ satisfying the centering condition:
\begin{align}\label{cen222}
\int_{H}\phi(t,x,y,z)\mu^x(\dif y)=0,\quad\forall t>0, x,z\in \cH,
\end{align}
let $\psi(t,x,y,z)$ solve the following Poisson equation
\begin{align}\label{psi1}
\cL_2(x,y)\psi(t,x,y,z)=-\phi(t,x,y,z),
\end{align}
where $t,x,z$ are regarded as parameters. Define
\begin{align}\label{ftp}
\overline{\delta \cF\cdot\nabla_z\psi}(t,x,z):=\int_{H}\nabla_z\psi(t,x,y,z).\delta \cF(x,y)\mu^x(\dif y).
\end{align}
We first establish the following weak fluctuation estimates for an appropriate integral functional of $(X_s^\eps,Y_s^{\eps},\cZ_s^{\eps})$ over the time interval $[0,t]$, which will play an important role in the proof of (\ref{nzz}).

\bl[Weak fluctuation estimates]\label{wfe2}
Let $T,\eta>0$, $x\in \cH^1$ and $y\in H$. Assume that ${\cF}\in C_b^{2,\eta}(\cH\times H,\cH^1)$ and $\cG\in C^{2,\eta}_B(\cH\times H,H).$ Then for any $t\in[0,T],$ $\phi\in C_{\ell}^{1,2,\eta,2}([0,T]\times\cH\times H\times \cH)$ satisfying (\ref{cen222}) and
 \begin{align}\label{as1}
|\p_t\phi(t,x,y,z)|&\leq C_0\big(1+\interleave x\interleave^2_1+\interleave z\interleave_1\no\\
&\qquad\quad+\interleave \cA x\interleave_0+\interleave \cA z\interleave_0\!\big)\big(1+\interleave x\interleave_1+\|y\|\big),
\end{align}
we have
\begin{align}
\mE\left(\int_0^t\phi(s,X_s^\eps,Y_s^{\eps},\cZ_s^{\eps})\dif s\right)\leq C_T\,\eps^{\frac{1}{2}},\label{we1}
\end{align}
and
\begin{align}
\mE\left(\frac{1}{\sqrt{\eps}}\int_0^t\phi(s,X_s^\eps,Y_s^{\eps},\cZ_s^{\eps})\dif s\right)\!-\!\mE\left(\int_0^t\overline{\delta \cF\cdot\nabla_z\psi}(s,X_s^{\eps},\cZ_s^{\eps})\dif s\right)\leq C_T\,\eps^{\frac{1}{2}},\label{we2}
\end{align}
where $C_T>0$ is a constant independent of $\eps,\eta$ and $n$.
\el
\begin{proof} The proof will be divided into two steps.

\vspace{1mm}
\noindent{\bf Step 1.}  We first prove estimate (\ref{we1}).
Applying It\^o's formula to $\psi(t,X_t^\eps,Y_t^{\eps},\cZ_t^\eps)$ and taking expectation, we have
\begin{align*}
\mE[\psi(t,X_t^\eps,Y_t^{\eps},\cZ_t^\eps)]&=\psi(0,x,y,0)+
\mE\left(\int_0^t(\partial_s+\mathcal{L}_1+\cL_3)\psi(s,X_s^\eps,Y_s^{\eps},\cZ_s^\eps)\dif s\right)\\
&\quad+\frac{1}{\eps}\mE\left(\int_0^t\mathcal{L}_2\psi(s,X_s^\eps,Y_s^{\eps},\cZ_s^\eps)\dif s\right),
\end{align*}
where $\cL_1$ and $\cL_3$ are defined by (\ref{L1}) and (\ref{333}), respectively.
Combining this with (\ref{psi1}), we obtain
\begin{align}\label{i}
&\mE\left(\int_0^t\phi(s,X_s^\eps,Y_s^{\eps},\cZ_s^{\eps})\dif s\right)
=\eps\mE\big[\psi(0,x,y,0)-\psi(t,X_t^\eps,Y_t^{\eps},\cZ_t^\eps)\big]\no\\
&\qquad+\eps\mE\left(\int_0^t\partial_s\psi(s,X_s^\eps,Y_s^{\eps},\cZ_s^{\eps})\dif s\right)+\eps\mE\left(\int_0^t\mathcal{L}_1\psi(s,X_s^\eps,Y_s^{\eps},\cZ_s^{\eps})\dif s\right)\no\\
&\qquad+\eps\mE\left(\int_0^t\mathcal{L}_3\psi(s,X_s^\eps,Y_s^{\eps},\cZ_s^{\eps})\dif s\right)=:\sum_{i=1}^4\sQ_i(t,\eps).
\end{align}
By Theorem \ref{PP} and Lemma {\ref{la41}}, we have
\begin{align*}
\sQ_1(t,\eps)\leq C_1\eps\mE\big(1+\interleave X_t^\eps\interleave_1+\|Y_t^{\eps}\|\big)\leq C_1\eps.
\end{align*}
For the second term, by using Theorem \ref{PP}, condition (\ref{as1}), Lemma \ref{la41}, (\ref{zne}) and (\ref{azne}), we get
\begin{align*}
\sQ_2(t,\eps)\leq C_2\bigg(\int_0^t\mE\big(1&+\interleave\cA X_s^\eps\interleave_0^2+\interleave\cA \cZ_s^\eps\interleave_0^2\\
&+\interleave X_s^\eps\interleave_1^4++\|Y_s^{\eps}\|^2+\interleave \cZ_s^\eps\interleave_1^2\big)\dif s\bigg)\leq C_2\,\eps.
\end{align*}
To treat the third term, since for each $t\in [0,T]$, $\phi(t,\cdot,\cdot,\cdot)\in C_{\ell}^{2,\eta,2}(\cH\times H\times\cH)$, by Theorem \ref{PP}, we have $\psi(t,\cdot,\cdot,\cdot)\in C_{\ell}^{2,2,2}(\cH\times H\times\cH)$, hence
\begin{align*}
|\cL_1\psi(t,X_t^\eps,Y_t^{\eps},\cZ_t^\eps)|&\leq |\langle \cA X_t^\eps+\cF(X_t^\eps,Y_t^\eps),D_x\psi(t,X_t^\eps,Y_t^{\eps},\cZ_t^\eps)\rangle_\cH|
\\&\quad+\frac{1}{2}
Tr((BQ_1^{\frac{1}{2}})(BQ_1^{\frac{1}{2}})^*)\|D^2_{x}\psi(t,X_t^\eps,Y_t^{\eps},\cZ_t^\eps)\|_{\sL(\cH\times\cH)}\\&
\leq C_3\big(1+\interleave\cA X_t^\eps\interleave_0^2+\interleave X_t^\eps\interleave_1^2+\|Y_t^{\eps}\|^2\big).
\end{align*}
As a result of  Lemmas \ref{la41} and \ref{la44}, we deduce that
\begin{align*}
\sQ_3(t,\eps)&\leq C_3\,\eps\mE\left(\int_0^t\Big(\interleave\cA X_s^\eps\interleave_0^2+\interleave X_s^\eps\interleave_1^2+\|Y_s^{\eps}\|^2\Big)\dif s\right)\leq C_3\,\eps.
\end{align*}
For the last term, we have
\begin{align*}
\sQ_{4}(t,\eps)&=\eps\mE\left(\int_0^t\<\cA \cZ_s^\eps,D_z\psi(s,X_s^\eps,Y_s^{\eps},\cZ_s^\eps)\>_\cH\dif s\right)\\&+\sqrt{\eps}\mE\left(\int_0^t\< \bar \cF(X_s^\eps)-\bar \cF(\bar X_s),D_z\psi(s,X_s^\eps,Y_s^{\eps},\cZ_s^\eps)\>_\cH\dif s\right)\\
&+\sqrt{\eps}\mE\left(\int_0^t\<\delta \cF(X_s^\eps,Y_s^{\eps}),D_z\psi(s,X_s^\eps,Y_s^{\eps},\cZ_s^\eps)\>_\cH\dif s\right)
.
\end{align*}
It follows from (\ref{azne}) and Lemma \ref{la41} again that
\begin{align*}
\sQ_{4}(t,\eps)&\leq C_4\sqrt{\eps}\mE\left(\int_0^t(1+\interleave \cA \cZ_s^\eps\interleave_0^2+\interleave X_s^\eps\interleave_1^2+\|Y_s^{\eps}\|^2)\dif s\right)
\leq C_4 \sqrt{\eps}.
\end{align*}
Combining the above inequalities with (\ref{i}), we get the desired result.

\vspace{1mm}
\noindent{\bf Step 2.} We proceed to prove estimate (\ref{we2}). By following exactly the same arguments as in  the proof of Step 1, we get that
\begin{align*}
&\mE\left(\frac{1}{\sqrt{\eps}}\int_0^t\phi(s,X_s^\eps,Y_s^{\eps},\cZ_s^{\eps})\dif s\right)\leq C_0\sqrt{\eps}+\sqrt{\eps}\mE\left(\int_0^t\mathcal{L}_3\psi(s,X_s^\eps,Y_s^{\eps},\cZ_s^{\eps})\dif s\right).
\end{align*}
For the last term, by definition (\ref{333}) we have
\begin{align*}
&\sqrt{\eps}\mE\left(\int_0^t\mathcal{L}_3\psi(s,X_s^\eps,Y_s^{\eps},\cZ_s^{\eps})\dif s\right)\\&=\sqrt{\eps}\mE\left(\int_0^t\<\cA \cZ_s^\eps,D_z\psi(s,X_s^\eps,Y_s^{\eps},\cZ_s^\eps)\>_\cH\dif s\right)\\
&+\mE\left(\int_0^t\< \bar \cF(X_s^\eps)-\bar \cF(\bar X_s),D_z\psi(s,X_s^\eps,Y_s^{\eps},\cZ_s^\eps)\>_\cH\dif s\right)\\
&+\mE\left(\int_0^t\<\delta \cF(X_s^\eps,Y_s^{\eps}),D_z\psi(s,X_s^\eps,Y_s^{\eps},\cZ_s^\eps)\>_\cH\dif s\right)
=:\sum\limits_{i=1}^3\sT_{i}(t,\eps).
\end{align*}
Using Lemma \ref{la41} and (\ref{azne}), we get
\begin{align*}
\sT_{1}(t,\eps)\leq C_1\sqrt{\eps}\mE\left(\int_0^t\interleave \cA \cZ_s^\eps\interleave_0(1+\interleave X_s^\eps\interleave_1+\|Y_s^{\eps}\|)\dif s\right)
\leq C_1 \sqrt{\eps}.
\end{align*}
According to H\"older's inequality, Lemma \ref{la41} and Theorem \ref{main1}, we have
\begin{align*}
\sT_{2}(t,\eps)\leq C_2\int_0^t\big(\mE\interleave X_s^{\eps}-\bar X_s\interleave_1^2\big)^{1/2}\big(1+\mE|\| X_s^\eps|\|_1^2+\mE\| Y_s^{\eps}\|^{2}\big)^{1/2}    \dif s\leq C_2\,\sqrt{\eps}.
\end{align*}
Thus, we deduce that
\begin{align*}
&\mE\left(\frac{1}{\sqrt{\eps}}\int_0^t\phi(s,X_s^\eps,Y_s^{\eps},\cZ_s^{\eps})\dif s\right)-\mE\left(\int_0^t\overline{\delta \cF\cdot\nabla_z\psi}(s,X_s^{\eps},\cZ_s^{\eps})\dif s\right)\no\\
&\leq\! C_3 \sqrt{\eps}+\!\mE\left(\int_0^t\!\!\big(\<\delta \cF(X_s^\eps,Y_s^{\eps}),D_z\psi(s,X_s^\eps,Y_s^{\eps},\cZ_s^\eps)\>_\cH-\overline{\delta \cF\cdot\nabla_z\psi}(s,X_s^{\eps},\cZ_s^{\eps})\big)\dif s\!\right),
\end{align*}
where $\overline{\delta \cF\cdot\nabla_z\psi}$ is defined by (\ref{ftp}). Note that the function
$$
\tilde \phi(t,x,y,z):=\<\delta \cF(x,y),D_z\psi(t,x,y,z)\>_1-\overline{\delta \cF\cdot\nabla_z\psi}(t,x,z)
$$
satisfies the centering condition (\ref{cen222}) and condition (\ref{as1}). Thus, using (\ref{we1}) directly, we obtain
$$\mE\left(\int_0^t\big(\<\delta \cF(X_s^\eps,Y_s^{\eps}),D_z\psi(s,X_s^\eps,Y_s^{\eps},\cZ_s^\eps)\>_\cH-\overline{\delta \cF\cdot\nabla_z\psi}(s,X_s^{\eps},\cZ_s^{\eps})\big)\dif s\right)\leq C_4\,\sqrt{\eps},$$
which completes the proof.
\end{proof}

Now, we are in the position to give:
\begin{proof}[Proof of estimate (\ref{nzz})]
 Fix $T>0.$ For any  $t\in[0,T]$ and $x,z\in \cH^1$, let
$$
\tilde u(t,x,z)=\bar u(T-t,x,z).
$$
It is easy to check that
$$
\tilde u(0,x,0)=\bar u(T,x,0)=\mE[\varphi(\bar \cZ_T)]\quad\text{and}\quad\tilde u(T,x,z)=\bar u(0,x,z)=\varphi(z).
$$
Applying It\^o's formula, by (\ref{kez}) we have
\begin{align*}
&\mE[\varphi(\cZ_T^{\eps})]-\mE[\varphi(\bar \cZ_T)]=\mE[\tilde u(T,\bar X_T,\cZ_T^{\eps})-\tilde u(0,x,0)]\\&
=\mE\left(\int_0^T\big(\p_t+\cL_1+\cL_3\big)\tilde u(t,X_t^{\eps},\cZ_t^{\eps})\dif t\right)\no\\
&=\mE\left(\int_0^T(\mathcal{L}_1-\mathcal{\bar L}_1)\tilde u(t,X_t^{\eps},\cZ_t^{\eps})\dif t\right)+\mE\left(\int_0^T(\mathcal{L}_3-\mathcal{\bar L}_3)\tilde u(t,X_t^{\eps},\cZ_t^{\eps})\dif t\right)\no\\
&=\mE\left(\int_0^T\langle \cF(X_t^\eps,Y_t^\eps)-\bar \cF(X_t^\eps), D_x\tilde u(t,X_t^{\eps},\cZ_t^{\eps}))\rangle_\cH \dif t\right)\\
&\quad+\mE\left(\int_0^T\Big\langle \frac{\bar \cF(X_t^\eps)-\bar \cF(\bar X_t)}{\sqrt{\eps}}-D_x\bar \cF( X_t^\eps) .\cZ_t^\eps, D_z\tilde u(t,X_t^{\eps},\cZ_t^{\eps}))\Big\rangle_\cH \dif t\right)\\
&\quad+\bigg[\mE\left(\frac{1}{\sqrt{\eps}}\int_0^T\langle \cF(X_t^\eps,Y_t^\eps)-\bar \cF(X_t^\eps), D_z\tilde u(t,X_t^{\eps},\cZ_t^{\eps}))\rangle_\cH \dif t\right)\\
&\qquad\qquad-\frac{1}{2}\mE\left(\int_0^T  Tr(D^2_{z}\tilde u(t,X_t^{\eps},\cZ_t^{\eps})\Sigma( X_t^\eps)\Sigma( X_t^\eps)^{*})\dif t\right)\bigg]:=\sum\limits_{i=1}^3\sN_i(T,\eps).
\end{align*}
 For the first term, recall that $\tilde\Psi$  solves the Poisson equation (\ref{poF1}) and define
$$
\psi(t,x,y,z):=\<\tilde\Psi(x,y),D_x\tilde u(t,x,z)\>_\cH.
$$
Since $\cL_2$ is an operator with respect to the $y$-variable, one can check that $\psi$ solves the following Poisson equation:
\begin{align*}
\mathcal{L}_2(x,y)\psi(t,x,y,z)=-\langle \delta \cF(x,y),D_x\tilde u(t,x,z)\rangle_\cH=:-\phi(t,x,y,z).
\end{align*}
It is obvious that $\phi$ satisfies the centering condition (\ref{cen222}). 	Furthermore,
by (\ref{utzx}) we get
\begin{align*}
|\p_t\phi(t,x,y,z)|&=\left|\<\delta \cF(x,y), \p_tD_{x}\bar u(T-t,x,z)\>_\cH\right|\\
&\leq C_1\big(1+\interleave x\interleave^2_1+\interleave z\interleave_1+\interleave \cA x\interleave_0+\interleave \cA z\interleave_0\big)\\
&\qquad\times\big(\interleave \delta \cF(x,y)\interleave_1+\interleave \cA \delta \cF(x,y)\interleave_0\big)\\
&\leq C_1\big(1+\interleave x\interleave^2_1+\interleave z\interleave_1+\interleave \cA x\interleave_0+\interleave \cA z\interleave_0\!\big)\big(1+\interleave x\interleave_1+\|y\|\big).
\end{align*}
Thus, it follows from (\ref{we1}) directly that
$$\sN_1(T,\eps)\leq C_1\sqrt{\eps}.
$$
To control the second term, by the mean value theorem, H\"older's inequality, Lemma \ref{bure}, Theorem \ref{main1} and (\ref{zne}) we deduce that for $\vartheta\in(0,1),$
\begin{align*}
\sN_2(T,\eps)&\leq\mE\bigg(\int_0^T\big|\big\langle[D_x\bar \cF(X_t^\eps+\vartheta(X_t^\eps-\bar X_t))\\
&\qquad\qquad\qquad-D_x\bar \cF( X_t^\eps) ].\cZ_t^\eps, D_z\tilde u(t,X_t^{\eps},\cZ_t^{\eps})\big\rangle_\cH\big|\dif t\bigg)\\
&\leq C_2\int_0^T\big(\mE\interleave X_t^\eps-\bar X_t\interleave _1^2\big)^{1/2}\big(\mE\interleave \cZ_t^\eps\interleave_1^2\big)^{1/2}\dif t\leq C_2\sqrt{\eps}.
\end{align*}
 For the last term, define
$$
\hat\psi(t,x,y,z):=\<\tilde\Psi(x,y),D_z\tilde u(t,x,z)\>_\cH.
$$
Then $\hat\psi$ solves the Poisson equation
\begin{align*}
\mathcal{L}_2(x,y)\hat\psi(t,x,y,z)=-\langle \delta \cF(x,y),D_z\tilde u(t,x,z)\rangle_\cH=:-\hat\phi(t,x,y,z).
\end{align*}
By exactly the same arguments as above, we have that $\hat\phi$ satisfies the centering condition (\ref{cen222}) and  condition (\ref{as1}). 	
Furthermore, by the definition of $\Sigma$ in (\ref{sst1}), we have
\begin{align*}
&\overline{\delta \cF\cdot\nabla_z\hat\psi}(t,x,z)=\int_{H}D_z\hat\psi(t,x,y,z).\delta \cF(x,y)\mu^x(\dif y)\\&=\int_{H}D_z^2\tilde u(t,x,z).(\tilde\Psi(x,y),\delta \cF(x,y))\mu^x(\dif y)=\frac{1}{2}Tr(D_z^2\tilde u(t,x,z)\Sigma(x)\Sigma^*(x)).
\end{align*}
Thus, it follows by (\ref{we2}) directly that
$$
\sN_3(T,\eps)\leq C_3\,\sqrt{\eps}.
$$
Combining the above computations, we get the desired result.
\end{proof}

\bigskip
\bibliographystyle{amsplain}

\end{document}